\documentclass[oneside,reqno,12pt]{amsart}

\usepackage{lutz-hdhog}

\title{Higher discrete homotopy groups of graphs}
\author{Bob Lutz}
\address{Mathematical Sciences Research Institute, 17 Gauss Way, Berkeley, CA 94720}
\email{boblutz@berkeley.edu}

\begin{document}
\begin{abstract}
This paper studies a discrete homotopy theory for graphs introduced by Barcelo et al. We prove two main results. First we show that if $G$ is a graph containing no 3- or 4-cycles, then the $n$th discrete homotopy group $A_n(G)$ is trivial for all $n\geq 2$. Second we exhibit for each $n\geq 1$ a natural homomorphism $\psi:A_n(G)\to \hc_n(G)$, where $\hc_n(G)$ is the $n$th discrete cubical singular homology group, and an infinite family of graphs $G$ for which $\hc_n(G)$ is nontrivial and $\psi$ is surjective. It follows that for each $n\geq 1$ there are graphs $G$ for which $A_n(G)$ is nontrivial.
\end{abstract}
\maketitle
\section{Introduction}
\label{sec:intro}

In \cite{barcelo2001} a new homotopy theory for simplicial complexes was introduced, motivated by a search for qualitative invariants in the study of complex systems and their dynamics \cite{barcelo2005}. Given a simplicial complex $K$, an integer $0\leq q\leq \dim K$ and a simplex $\sigma_0\in K$ of dimension at least $q$, one defines a family of groups
\begin{equation}
A_n^q(K,\sigma_0),\quad n\geq 1,
\label{eq:introgrp}
\end{equation}
called the \emph{discrete homotopy groups} of $K$. In contrast to classical homotopy theory, the groups \eqref{eq:introgrp} are defined combinatorially.


For $n=1$, the group \eqref{eq:introgrp} is called the \emph{discrete fundamental group}, and is well understood. There is a discrete analog of the Seifert-van Kampen theorem, and in fact one can construct a cell complex $X$ such that
\begin{equation}
A_1^q(K,\sigma_0)\cong \pi_1(X,x_0),
\label{eq:compcomp}
\end{equation}
where $\pi_1(X,x_0)$ is the classical fundamental group for some $x_0\in X$. Results like these have enabled computations of $A_1^q(K,\sigma_0)$ for interesting simplicial complexes, including Coxeter complexes of finite Coxeter groups \cite{barcelo2011, barcelo2008}.

For $n>1$, the situation is less clear. The basic tools for computing classical higher homotopy groups, already a difficult problem in general, have no known discrete analogs. In \cite{babson2006} a higher-dimensional version of \eqref{eq:compcomp} was obtained, assuming a ``plausible'' cubical analog of the simplicial approximation theorem. Despite its plausibility, however, the proposed approximation theorem has resisted all attempts at a proof.

This paper studies the problem of computing higher discrete homotopy groups. To begin, we  reduce this problem to considering only graphs instead of general simplicial complexes. Given a graph $G$ and a vertex $v_0\in G$, one can define a family of groups $A_n(G,v_0)$ similarly to \eqref{eq:introgrp}. For any triple $(K,q,\sigma_0)$ as above, there is a pair $(G,v_0)$ such that
\[A_n^q(K,\sigma_0)\cong A_n(G,v_0).\]
Namely, $G$ is the graph whose vertices correspond to the maximal simplices of $K$ of dimension at least $q$, and whose edges correspond to the pairs of such simplices sharing a $q$-face. We can therefore restrict our attention to the groups $A_n(G,v_0)$ without loss of generality. In particular, we will consider only connected graphs, so we can ignore the base vertex and simply write $A_n(G)=A_n(G,v_0)$.

We prove two main results. The first is motivated by a singular homology theory for graphs, introduced in \cite{barcelo2014} as a companion to the discrete homotopy theory described above. The relevant groups $\hc_n(G)$, called the \emph{discrete singular cubical homology groups} of $G$, are defined combinatorially. It can be shown that if $G$ contains no 3- or 4-cycles, then $\hc_n(G)$ is trivial for all $n\geq 2$. This was the main result of \cite{barcelo2019}. We prove a similar theorem for discrete homotopy.

\begin{thm}
If $G$ contains no 3- or 4-cycles, then $A_n(G)$ is trivial for all $n\geq 2$.
\label{thm:mainthm1}
\end{thm}

If $G$ contains no 3- or 4-cycles, then it can be shown that $A_1(G)\cong \pi_1(G)$, where we regard $G$ as both a graph and a topological space. Thus Theorem \ref{thm:mainthm1} gives a complete picture of the discrete homotopy of such graphs.

Our second result makes progress toward a discrete Hurewicz theorem. Recall that the classical Hurewicz theorem \cite{hurewicz1935} describes, for any path-connected topological space $X$, a map from $\pi_n(X)$ to the $n$th homology group of $X$. If $n=1$, then this map is surjective, and its kernel is the commutator subgroup $[\pi_1(X),\pi_1(X)]$. If $n>1$ and $\pi_k(X)$ is trivial for all $k<n$, then the map is an isomorphism.

In \cite{barcelo2014} it was shown that there is a surjective map
\begin{equation}
A_1(G)\to \hc_1(G)
\label{eq:dischur}
\end{equation}
with kernel $[A_1(G),A_1(G)]$, giving a discrete Hurewicz theorem in dimension 1. We generalize the map \eqref{eq:dischur} to higher dimensions and show that it is surjective in a large number of cases.

\begin{thm}
For any graph $G$ and positive integer $n$, there is a natural homomorphism
\begin{equation}
\psi: A_n(G)\to \hc_n(G).
\label{eq:intropsi}
\end{equation}
For each $n\geq 1$, there is an infinite family of graphs $G$ for which $\psi$ is surjective and $\hc_n(G)$ is nontrivial.
\label{thm:mainthm2}
\end{thm}

An important consequence of Theorem \ref{thm:mainthm2} is the existence of graphs $G$ with nontrivial groups $A_n(G)$ for $n\geq 2$. To date, no examples of such graphs have appeared in the literature.

\begin{cor}
For each $n\geq 1$ there are graphs $G$ for which $A_n(G)$ is nontrivial.
\end{cor}

The paper is organized as follows. In Section \ref{sec:basic} we give preliminary definitions and results concerning discrete homotopy theory. In Section \ref{sec:tsfree} we prove Theorem \ref{thm:mainthm1}. In Section \ref{sec:conesus} we describe the basic constructions needed for Theorem \ref{thm:mainthm2}. In Section \ref{sec:hurewicz} we recall the definition of discrete cubical singular homology groups and prove Theorem \ref{thm:mainthm2}. In Section \ref{sec:openq} we suggest directions for future work.
\section{Basic definitions}
\label{sec:basic}

\subsection{Discrete homotopy}

In this paper, the term \emph{homotopy} refers to the discrete homotopy theory for graphs introduced in \cite{barcelo2001}. By a \emph{graph} we will mean one that is simple, connected and undirected. We write $v\in \g$ if $v$ is a vertex of a graph $\g$. If $\g_1$ and $\g_2$ are graphs, then we write $f:\g_1\to \g_2$ when $f$ is a function from the vertex set of $\g_1$ to the vertex set of $\g_2$. If two vertices $u,v\in G$ are adjacent or equal, then we write $u\simeq v$.

\begin{mydef}
A function $f:\g_1\to \g_2$ is a \emph{graph map} if $f(u)\simeq f(v)$ whenever $u\simeq v$.
\end{mydef}

Let $\Z$ denote the graph with vertex set $\{\ldots,-1,0,1,\ldots\}$ and an edge $ij$ if and only if $|i-j|=1$. Given an integer $m\geq 0$, we write $I_m$ for the subgraph of $\Z$ induced by $\{0,\ldots,m\}$. We let $G_1\times G_2$ denote the Cartesian product of graphs $G_1$ and $G_2$. The following notion of homotopy defines an equivalence relation on graph maps $G_1\to G_2$.

\begin{mydef}
Two graph maps $f,g:\g_1\to \g_2$ are \emph{homotopic} if for some $m$ there is a graph map $h:\g_1\times I_m\to \g_2$ such that $h(-,0)=f$ and $h(-,m)=g$. The map $h$ is a \emph{homotopy} from $f$ to $g$. We write $h_i = h(-,i)$ for all $i$, so $h_0=f$ and $h_m=g$.
\end{mydef}

Let $\Z^n$ denote the $n$-fold Cartesian product of $\Z$, whose vertices are $n$-tuples of integers. For any $r\geq 0$, let $\znger\subseteq \Z^n$ be the subgraph induced by all vertices $x\in \Z^n$ with $|x_i|\geq r$ for some $i$.

Given a function $f:G_1\to G_2$, and subgraphs $H_1\subseteq G_1$ and $H_2\subseteq G_2$, we write
\begin{equation}
f:(\g_1,H_1)\to (\g_2,H_2)
\end{equation}
if $v\in H_1$ implies that $f(v)\in H_2$. We will abuse this notation by writing
\begin{equation}
f:(\Z^n,\partial \Z^n)\to (\g,v_0)
\end{equation}
if there exists an $r\geq 0$ such that $f$ is a function $(\Z^n,\znger)\to (G,v_0)$. The minimum such $r$ is called the \emph{radius} of $f$. We emphasize that $\partial \Z^n$ is not defined by itself, but is shorthand for some $\znger$. The following notion of homotopy, more restrictive than the last, defines an equivalence relation on graph maps $(\Z^n,\partial \Z^n)\to (\g,v_0)$.

\begin{mydef}
Two graph maps $f,g:(\Z^n,\partial \Z^n)\to (\g,v_0)$, are \emph{based homotopic} if there is a homotopy $h:\Z^n\times I_m \to \g$ from $f$ to $g$ such that $h_i$ is a graph map $(\Z^n,\partial \Z^n)\to (G,v_0)$ for all $i$. The function $h$ is a \emph{based homotopy} from $f$ to $g$.
\end{mydef}

\begin{mydef}
Fix $v_0\in\g$. Let $A_n(\g)$ denote the set of based homotopy classes $[f]$ of graph maps $f:(\Z^n,\partial \Z^n)\to (\g,v_0)$. We endow $A_n(\g)$ with a group structure as follows. If $f$ and $g$ are graph maps $(\Z^n,\partial \Z^n)\to (\g,v_0)$ of radii $r_f$ and $r_g$, respectively, then the product $[f]\cdot [g]$ in $A_n(\g)$ is the based homotopy class of the map $p:(\Z^n,\partial \Z^n)\to (\g,v_0)$ given by
\begin{equation}
  p(x_1,\ldots,x_n) =
  \begin{cases}
    f(x_1,\ldots,x_n)&\mbox{if } x_1\leq r_f\\
    g(x_1-(r_f+r_g),x_2,\ldots,x_n)&\mbox{if } x_1>r_f.
  \end{cases}
\label{eq:concat}
\end{equation}
The identity of this operation is the based homotopy class of the constant map $\Z^n\to v_0$. We call $A_n(\g)$ the \emph{$n$th discrete homotopy group} of $\g$.
\end{mydef}

\begin{figure}[ht]
\begin{tikzpicture}[scale=3]
\draw[thick,fill=black!15] (-0.3,-0.3) rectangle (0.3,0.3);
\draw[thick,fill=black!06] (0.3,-0.4) rectangle (1.1,0.4);
\node () at (0,0) {$f$};
\node () at (0.7,0) {$g$};
\end{tikzpicture}
\caption{Illustrating the group operation in $A_2(G)$.}
\end{figure}
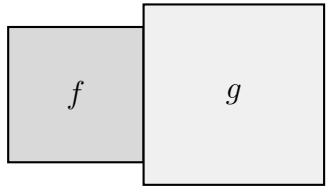

\begin{remark}
In \cite{barcelo2001}, working in the more general setting of simplicial complexes, it is shown that the operation in $A_n(G)$ is well defined, that it satisfies the group axioms, and that the definition of $A_n(G)$ is independent of the base vertex $v_0$ when $G$ is connected. It is also shown that the groups $A_n(G)$ are abelian for all $n\geq 2$. We will not prove these facts here.
\end{remark}

\begin{prop}[{\cite[Proposition 5.12]{barcelo2001}}]
Let $G$ be a graph, and let $X_G$ be the cell complex obtained by regarding $G$ as a 1-complex and attaching a 2-cell along the boundary of each 3- and 4-cycle. We have $A_1(G)\cong \pi_1(X_G)$.
\label{prop:cwcomp}
\end{prop}

\begin{eg}
Let $G$ be a cycle graph of length $m$, i.e. $G=\Z_m$. If $m=3$ or 4, then $X_G$ is simply connected, so $A_1(G)$ is trivial by Proposition \ref{prop:cwcomp}. If $m\geq 5$, then $X_G\approx S^1$, so $A_1(G)\cong \Z$. Label the vertices of $G$ in a cycle as $v_0,\ldots,v_{m-1}$. An explicit generator of $A_1(G)$ is $[f]$, where $f$ is given by $f(i)=v_i$ if $0\leq i<m$ and $f(i)=v_0$ otherwise.

More generally, suppose that $G$ contains no 3- or 4-cycles. For each cycle $Z$ of $G$ we obtain a generator $[f_Z]$ of $A_1(Z)$ as above. A set of generators of $A_1(G)$, although not necessarily a minimal one, is the set of $[f_Z]$ for all $Z$.
\label{eg:babyfg}
\end{eg}

\begin{eg}
Let $K$ be a finite simplicial 2-complex. Suppose, for any vertices $i$, $j$ and $k$, that if the edges $ij$, $ik$ and $jk$ all belong to $K$, then so does the triangle $ijk$. Suppose also that the 1-skeleton $K^1$ is a \emph{chordal} graph, i.e. that if a cycle of $G$ is an induced subgraph, then it contains exactly 3 vertices. These conditions ensure that if $G=K^1$, then $X_G=K$ in the notation of Proposition \ref{prop:cwcomp}, so $A_1(G)\cong \pi_1(K)$. For example, if $K$ is the triangulation of the real projective plane in Figure \ref{fig:pp}, then $A_1(G)\cong \Z/2\Z$. This example generalizes the remark following \cite[Proposition 5.12]{barcelo2001}.

\begin{figure}[ht]
\begin{tikzpicture}[scale=0.75]
\draw[fill=black!10] (0,0) rectangle (5,5);
\foreach\x in {0,...,5} {
\draw (\x,0) -- (0,\x);
\draw (5-\x,5) -- (5,5-\x);
\draw (0,\x) -- (5,\x);
\draw (\x,0) -- (\x,5);
}
\draw[ultra thick,magenta] (0,0) -- (0,1);
\draw[ultra thick,magenta] (5,5) -- (5,4);
\draw[ultra thick,cyan] (0,5) -- (1,5);
\draw[ultra thick,cyan] (5,0) -- (4,0);
\draw[ultra thick,magenta,->] (0,1) -- (0,5.05);
\draw[ultra thick,magenta,->] (5,4) -- (5,-0.05);
\draw[ultra thick,cyan,->] (1,5) -- (5.05,5);
\draw[ultra thick,cyan,->] (4,0) -- (-0.05,0);
\end{tikzpicture}
\caption{A triangulation of the projective plane.}
\label{fig:pp}
\end{figure}
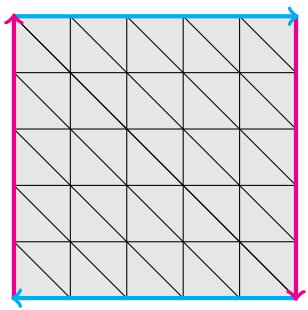
\end{eg}

\subsection{Contractibility}

We define a notion of a contractible graph as a discrete analog of a contractible topological space. We show that all discrete homotopy groups of a contractible graph are trivial. This fact will be used in proving the main results.

\begin{mydef}
A \emph{contraction} of $\g$ is a homotopy from the identity map $\g\to \g$ to a constant map. A graph $\g$ is \emph{contractible} if it admits a contraction.
\end{mydef}

By our definition, infinite graphs are not contractible. We caution that this notion of contractibility does not imply that $\pi_1(\g)$ is trivial when $\g$ is regarded as a topological space. For instance, the next example shows that the square graph $\Z_4$ is contractible, but $\pi_1(\Z_4)\cong\Z$.

\begin{eg}[Grid graphs]
A \emph{grid graph} is any subgraph $L$ of $\Z^n$ induced by a vertex set of the form
\[\{a_1,\ldots,b_1\}\times \cdots \times \{a_n,\ldots,b_n\}\]
for some $a_i,b_i\in \Z$ with $a_i\leq b_i$ for all $i$. The \emph{boundary} of $L$ is the subgraph $\partial L$ induced by all vertices $x\in L$ with $x_i\in \{a_i,b_i\}$ for some $i$. Its complement $L^\circ=L\setminus \partial L$ is the \emph{interior} of $L$. Let $m=\max_i b_i-a_i$, and let $c:L\times I_m\to L$ be given by
\[c(v,i) = (\max\{v_1-i,0\},\ldots,\max\{v_n-i,0\})\]
for all $(v,i)\in L\times I_m$. It is routine to check that $c$ is a homotopy. In particular, $c_0$ is the identity map $L\to L$, and $c_m$ is the constant map $L\to (0,\ldots,0)$. Hence $L$ is contractible.
\label{eg:gridcon}
\end{eg}

\begin{eg}[Finite trees]
Suppose that $T$ is a finite tree, and let $v_0\in T$. For every $v\in T\setminus v_0$, let $\rho(v)$ denote the unique neighbor of $v$ that lies on the simple path from $v_0$ to $v$. Set $\rho(v_0)=v_0$. This defines a graph map $\rho:\g\to \g$. Let $m$ be the diameter of $\g$. We define a function $c:\g\times I_m\to \g$ by setting
\[c(v,i) = \begin{cases}v&\mbox{if }i=0 \\ \rho(c(v,i-1))&\mbox{if }1\leq i\leq m\end{cases}\]
for all $v\in T$. It is routine to check that $c$ is a homotopy. In particular, $c_0$ is the identity map $T\to T$ and $c_m=v_0$, so $G$ is contractible.

\label{eg:acyclic}
\end{eg}

\begin{prop}
If $\g$ is contractible, then $A_n(\g)$ is trivial for all $n$.

\begin{proof}
Let $c:G\times I_m\to G$ be a contraction of $G$ to a vertex $v_0$, and let $f:(\Z^n,\partial \Z^n)\to (G,v_0)$ be a graph map of radius $r$. If $c(v_0,-)$ is constant, then there is an obvious based homotopy $\Z^n\times I_m \to G$ from $f$ to the constant map $v_0$, given by
\begin{equation}
(x,i)\mapsto c(f(x),i)
\label{eq:easyhom}
\end{equation}
for all $(x,i)\in \Z^n\times I_m$. However, it need not be the case that $c(v_0,-)$ is constant. In general, the map \eqref{eq:easyhom} is a homotopy but not a based homotopy, and hence does not preserve the equivalence class $[f]\in A_n(G)$.

To get around this, we define a based homotopy $h:\Z^n\times I_{2m}\to G$ from $f$ to $v_0$ in two steps. We first define $h_i$ for $i=0,\ldots,m$, and then for $i=m+1,\ldots,2m$. Let $\znler\subseteq \Z^n$ denote the grid graph
\[\znler = \{-r,\ldots,r\}\times \cdots \times \{-r,\ldots,r\}.\]
Let $B_0=\partial \znler$. For $k\geq 1$, let $B_k$ denote the boundary of the grid graph whose interior is $\znler\cup B_0\cup \cdots \cup B_{k-1}$. Note that $\{\znler,B_1,B_2,\ldots\}$ partitions $\Z^n$. Let $h_0=f$. For $i=1,\ldots,m$ let $h_i(x)=c(f(x),i)$ for all $x\in \znler$, and let $h_i$ be constant on $B_k$ for each $k\geq 1$, taking the same value that $h_{i-1}$ takes on $B_{k-1}$. The result is that $h_m$ is constant on $\znler$ and each $B_k$. 

\begin{figure}[ht]
\centering
\begin{tikzpicture}[scale=2.1]
\def\a{0.22}
\def\b{0.5}
\coordinate (a) at (0,0);
\draw[fill=cyan] ($(a)+(\b/4,2*\a)$) circle (1pt);
\draw[fill=cyan] ($(a)+(\b,2*\a)$) circle (1pt);
\draw[fill=cyan] ($(a)+(5*\b,2*\a)$) circle (1pt);
\draw[fill=cyan] ($(a)+(5*\b,\a)$) circle (1pt);
\draw[fill=cyan] ($(a)+(5*\b,0)$) circle (1pt);
\draw[fill=goldenyellow] ($(a)+(\b/4,\a)$) circle (1pt);
\draw[fill=goldenyellow] ($(a)+(\b,\a)$) circle (1pt);
\draw[fill=goldenyellow] ($(a)+(2*\b,\a)$) circle (1pt);
\draw[fill=goldenyellow] ($(a)+(2*\b,2*\a)$) circle (1pt);
\draw[fill=goldenyellow] ($(a)+(4*\b,2*\a)$) circle (1pt);
\draw[fill=goldenyellow] ($(a)+(4*\b,\a)$) circle (1pt);
\draw[fill=goldenyellow] ($(a)+(4*\b,0)$) circle (1pt);
\draw[fill=magenta] ($(a)+(\b/4,0)$) circle (1pt);
\draw[fill=magenta] ($(a)+(\b,0)$) circle (1pt);
\draw[fill=magenta] ($(a)+(2*\b,0)$) circle (1pt);
\draw[fill=magenta] ($(a)+(3*\b,0)$) circle (1pt);
\draw[fill=magenta] ($(a)+(3*\b,0)$) circle (1pt);
\draw[fill=magenta] ($(a)+(3*\b,\a)$) circle (1pt);
\draw[fill=magenta] ($(a)+(3*\b,2*\a)$) circle (1pt);
\draw (0,\a/2) -- (11*\b/2,\a/2);
\draw (0,3*\a/2) -- (11*\b/2,3*\a/2);
\draw (0,5*\a/2) -- (11*\b/2,5*\a/2);
\draw (\b/2,-\a/2) -- (\b/2,3.75*\a);
\draw (3*\b/2,-\a/2) -- (3*\b/2,3.75*\a);
\draw (5*\b/2,-\a/2) -- (5*\b/2,3.75*\a);
\draw (7*\b/2,-\a/2) -- (7*\b/2,3.75*\a);
\draw (7*\b/2,-\a/2) -- (7*\b/2,3.75*\a);
\draw (9*\b/2,-\a/2) -- (9*\b/2,3.75*\a);
\node[above] at (\b,5*\a/2) {0};
\node[above] at (2*\b,5*\a/2) {1};
\node[above] at (3*\b,5*\a/2) {2};
\node[above] at (4*\b,5*\a/2) {3};
\node[above] at (5*\b,5*\a/2) {4};
\draw (\b/2,5*\a/2) -- (0,3.75*\a);
\node at (\b/3,3.5*\a) {\small $i$};
\node at (\b/7,2.8*\a) {\small $v$};

\def\c{-3}
\def\r{1}
\draw (\c,3*\a/2) node[draw,circle,fill=cyan,scale=0.7] {} -- (\c+\r,3*\a/2) node[draw,circle,fill=goldenyellow,scale=0.7] {} -- (\c+2*\r,3*\a/2) node[draw,circle,fill=magenta,scale=0.7] {};
\end{tikzpicture}
\caption{A graph $G$ and a table of values of a contraction $c:G\times I_4\to G$.}
\label{fig:ctable}
\end{figure}
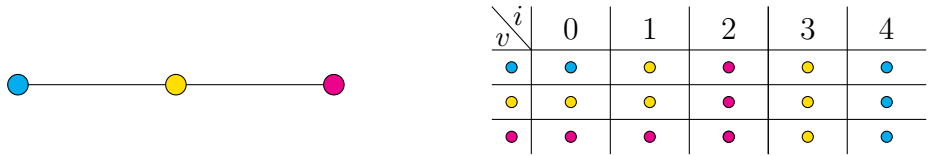

For example, let $G$ be the graph on the left side of Figure \ref{fig:ctable}, whose vertices are labeled by colors. Let $c:G\times I_4\to v_0$ be the contraction of $G$ with values $c(v,i)$ given in the table on the right side of Figure \ref{fig:ctable}, so that $m=4$. We can represent a graph map $\Z^2\to \g$ as an infinite 2-dimensional array of colored dots, where the color of the dot at $x\in \Z^2$ corresponds to the value of the map at $x$. Several such maps appear in Figure \ref{fig:dots2}, where we have restricted our attention to the domain $\Z^2_{\leq 5}$. We assume that each of these maps takes the color blue everywhere outside the pictured region. Let $f$ be the map on the top left of Figure \ref{fig:dots2}, and note that $r=2$ is the radius of $f$. Thus $\znler$ is outlined by dashes. The figure illustrates the maps $h_i$ from above for $i=0,\ldots,4$.

\begin{figure}[ht]
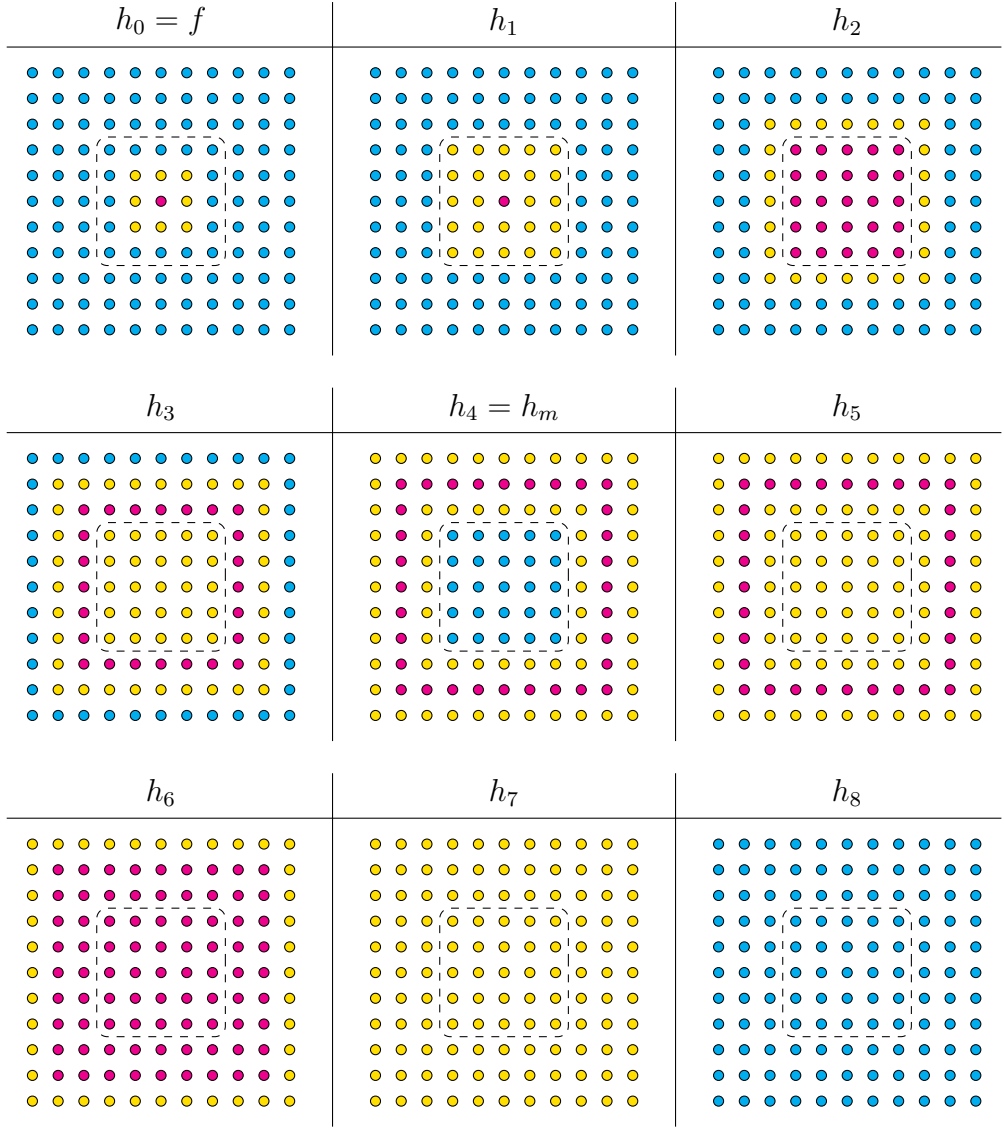

\centering
\include{tikz/dots2}
\caption{The stages of a homotopy $h:\Z^2\times I_8\to G$.}
\label{fig:dots2}
\end{figure}

Returning to the general case, we define $h_i$ for the remaining values of $i$. Suppose that $i\in \{m+1,\ldots,2m\}$. Let $h_i$ be constant on $\znler\cup B_{m+1}\cup\cdots\cup B_{i-m}$, taking the same value that $h_m$ takes on $B_{i-m}$. For $k>i-m$, let $h_i$ equal $h_m$ on $B_k$. Continuing the example from Figure \ref{fig:ctable}, these maps $h_i$ are illustrated in Figure \ref{fig:dots2} for $i=5,\ldots,8$. In the general case, it is routine to verify that $h:\Z^n\times I_{2m} \to G$ is a based homotopy; in particular, we have $h_0=f$ and $h_{2m}=v_0$. Hence $[f]=0$ in $A_n(G)$, as desired.
\end{proof}
\label{prop:contriv}
\end{prop}

If there exists a contraction $c$ of $G$ to $v_0$ such that $c(v_0,-)$ is constant, then we say that $G$ \emph{deformation retracts} onto $v_0$. In classical topology, there are examples of contractible spaces that do not deformation retract onto a point (see, e.g., \cite[Chapter 0, Exercise 6]{hatcher2002} and \cite[Section 2]{baillif2008}). However, these examples rely on behaviors that have no analogs in the discrete theory. We are led to ask the following.

\begin{q}
Suppose that $G$ is contractible. Does $G$ deformation retract onto a vertex?
\end{q}
\section{Triangle- and square-free graphs}
\label{sec:tsfree}

We prove Theorem \ref{thm:mainthm1}. If $G$ contains no 3- or 4-cycles, then $A_1(G)\cong \pi_1(G)$ by Proposition \ref{prop:cwcomp}, so we obtain a complete description of the discrete homotopy groups in this case.

\begin{mydef}
A \emph{path} in $\g$ is a sequence $(p_1,\ldots,p_k)$ of vertices of $G$ such that $p_i\simeq p_{i+1}$ for all $i=1,\ldots,k-1$.
\end{mydef}

\begin{proof}[Proof of Theorem \ref{thm:mainthm1}]
Suppose that $\g$ contains no 3- or 4-cycles, and set $n\geq 2$. Let $f:(\Z^n,\partial\Z^n)\to (\g,v_0)$ be a graph map of radius $r$. We show that $[f]=0$ in $A_n(G)$. For this proof only, consider $\g$ as directed with edge set $E$, so that if $(u,v)\in E$, then $(v,u)\notin E$. Let $F_E$ be the free group on $E$, whose elements are the identity 1 and reduced words in the letters $e$ and $e^{-1}$ for all $e\in E$.

We define a function $\tau$ from the set of paths in $\Z^n$ of length at least 2 to $F_E$ as follows. Given adjacent vertices $x$ ans $y$ of $\Z^n$, define an element $\tau(x,y)$ of $F_E$ by
\[\tau(x,y)=\begin{cases} (f(x),f(y))&\mbox{if }(f(x),f(y))\in E\\ (f(y),f(x))^{-1}&\mbox{if } (f(y),f(x))\in E\\ 1&\mbox{if }f(x)=f(y).\end{cases}\]
Given a path $P=(p_0,\ldots,p_\ell)$ in $\Z^n$ with $\ell>1$, let
\[\tau(P)=\tau(p_0,p_1)\tau(p_1,p_2)\cdots \tau(p_{\ell-1},p_\ell).\]

Let $g:\Z^n\to \g$ be given by
\begin{equation}
g(x)=\tau(P),
\label{eq:gfunc}
\end{equation}
where $P$ is any path in $\Z^n$ beginning in $\znger$ and ending with $x$. For example, consider the graph $G$ and the graph map $f:(\Z^2,\partial \Z^2)\to (G,a)$ illustrated in Figure \ref{fig:gpath}, where $r=4$ and $\znger$ consists of all vertices outside the dashed line. Let $x=(0,-2)$, so that $f(x)=c$. The three highlighted paths from $\znger$ to $x$ give three ways to compute $g(x)$, all resulting in the same value. Using the green path, clearly $g(x) = (a,b)(b,c)$. For the blue path, we have
\[g(x) = (a,b)(b,c)(c,d)(c,d)^{-1}(b,c)^{-1}(b,c) = (a,b)(b,c).\]
A similar computation with the red path gives the same result. Returning to the general case, it is not clear {\textit a priori} that $g$ is well defined. We will briefly assume that $g$ is well defined in order to finish the proof. We will then prove our assumption.

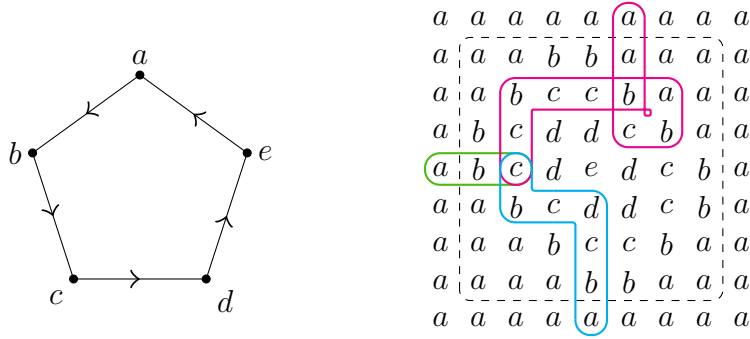
\begin{figure}[ht]
\centering
\begin{tikzpicture}[scale=1]
\tikzset{->-/.style={decoration={
			markings,
			mark=at position .5 with {\arrow[scale=2]{>}}},postaction={decorate}}}
\draw[dashed,rounded corners] (0.5/2,0.5/2) rectangle (7.5/2,7.5/2);
\foreach\x in {0,...,8} \node at (\x/2,0) {$a$};
\foreach\x in {0,...,8} \node at (\x/2,8/2) {$a$};
\foreach\x in {1,...,7} \node at (0,\x/2) {$a$};
\foreach\x in {1,...,7} \node at (8/2,\x/2) {$a$};
\node at (1/2,1/2) {$a$};
\node at (2/2,1/2) {$a$};
\node at (3/2,1/2) {$a$};
\node at (4/2,1/2) {$b$};
\node at (5/2,1/2) {$b$};
\node at (6/2,1/2) {$a$};
\node at (7/2,1/2) {$a$};

\node at (1/2,2/2) {$a$};
\node at (2/2,2/2) {$a$};
\node at (3/2,2/2) {$b$};
\node at (4/2,2/2) {$c$};
\node at (5/2,2/2) {$c$};
\node at (6/2,2/2) {$b$};
\node at (7/2,2/2) {$a$};

\node at (1/2,3/2) {$a$};
\node at (2/2,3/2) {$b$};
\node at (3/2,3/2) {$c$};
\node at (4/2,3/2) {$d$};
\node at (5/2,3/2) {$d$};
\node at (6/2,3/2) {$c$};
\node at (7/2,3/2) {$b$};

\node at (1/2,4/2) {$b$};
\node at (2/2,4/2) {$c$};
\node at (3/2,4/2) {$d$};
\node at (4/2,4/2) {$e$};
\node at (5/2,4/2) {$d$};
\node at (6/2,4/2) {$c$};
\node at (7/2,4/2) {$b$};

\node at (1/2,5/2) {$b$};
\node at (2/2,5/2) {$c$};
\node at (3/2,5/2) {$d$};
\node at (4/2,5/2) {$d$};
\node at (5/2,5/2) {$c$};
\node at (6/2,5/2) {$b$};
\node at (7/2,5/2) {$a$};

\node at (1/2,6/2) {$a$};
\node at (2/2,6/2) {$b$};
\node at (3/2,6/2) {$c$};
\node at (4/2,6/2) {$c$};
\node at (5/2,6/2) {$b$};
\node at (6/2,6/2) {$a$};
\node at (7/2,6/2) {$a$};

\node at (1/2,7/2) {$a$};
\node at (2/2,7/2) {$a$};
\node at (3/2,7/2) {$b$};
\node at (4/2,7/2) {$b$};
\node at (5/2,7/2) {$a$};
\node at (6/2,7/2) {$a$};
\node at (7/2,7/2) {$a$};

\def\a{0.21}
\draw[thick,kellygreen,rounded corners=2mm] (-\a,4/2+\a) -- (2/2+\a,4/2+\a) -- (2/2+\a,4/2-\a) -- (-\a,4/2-\a) -- cycle;
\draw[thick,magenta,rounded corners=2mm] (5/2+\a,8/2) -- (5/2+\a,8/2+\a) -- (5/2-\a,8/2+\a) -- (5/2-\a,5/2-\a) -- (6/2+\a,5/2-\a) -- (6/2+\a,6/2+\a) -- (2/2-\a,6/2+\a) -- (2/2-\a,4/2-\a) -- (2/2+\a,4/2-\a) -- (2/2+\a,4/2);
\draw[thick,magenta,rounded corners=0.2mm] (2/2+\a,4/2) -- (2/2+\a,6/2-\a) -- (6/2-\a,6/2-\a) -- (6/2-\a,5/2+\a) -- (5/2+\a,5/2+\a) -- (5/2+\a,8/2);
\draw[thick,cyan,rounded corners=2mm] (4/2-\a,0) -- (4/2-\a,-\a) -- (4/2+\a,-\a) -- (4/2+\a,3/2+\a) -- (3/2,3/2+\a);
\draw[thick,cyan, rounded corners=0.4mm] (3/2,3/2+\a) -- (2/2+\a,3/2+\a) -- (2/2+\a,4/2);
\draw[thick,cyan,rounded corners=2mm] (2/2+\a,4/2) -- (2/2+\a,4/2+\a) -- (2/2-\a,4/2+\a) -- (2/2-\a,3/2-\a) -- (3/2,3/2-\a);
\draw[thick,cyan, rounded corners=0.4mm] (3/2,3/2-\a) -- (4/2-\a,3/2-\a) -- (4/2-\a,0);

\def\b{-4}
\def\r{1.5}
\coordinate (a) at ($(\b,3.5/2)+(90:\r)$);
\coordinate (b) at ($(\b,3.5/2)+(162:\r)$);
\coordinate (c) at ($(\b,3.5/2)+(234:\r)$);
\coordinate (d) at ($(\b,3.5/2)+(306:\r)$);
\coordinate (e) at ($(\b,3.5/2)+(18:\r)$);
\draw[->-] (a) node[above] {$a$} -- (b) node[left] {$b$};
\draw[->-] (b) -- (c) node[below left] {$c$};
\draw[->-] (c) -- (d) node[below right] {$d$};
\draw[->-] (d) -- (e) node[right] {$e$};
\draw[->-] (e) -- (a);
\draw[fill=black] (a) circle (1.5pt);
\draw[fill=black] (b) circle (1.5pt);
\draw[fill=black] (c) circle (1.5pt);
\draw[fill=black] (d) circle (1.5pt);
\draw[fill=black] (e) circle (1.5pt);
\end{tikzpicture}
\caption{A graph $G=\Z_5$ and a graph map $f:(\Z^2,\partial \Z^2)\to (G,a)$.}
\label{fig:gpath}
\end{figure}

Assume that $g$ is well defined, and consider the Cayley graph $\Gamma(F_E,E)$ of $F_E$ with generating set $E$. Let $\g_E$ be the subgraph of $\Gamma(F_E,E)$ induced by
\[\{\tau(P): P\mbox{ is a path in }\Z^n\mbox{ beginning in }\znger\}.\]
Note that $\g_E$ is a finite tree, since it is a finite connected subgraph of the tree $F_E$. The distance between any two vertices $u$ and $v$ in $\g_E$ is given by $|u^{-1}v|$, where if $w\in F_E$, then $|w|$ denotes the length of the reduced word representing $w$. If $x\simeq y$ in $\Z^n$, then $|g(x)^{-1}g(y)|=|\tau(x,y)|\leq 1$, so $g(x)\simeq g(y)$ in $\g_E$. Hence $g:(\Z^n,\partial\Z^n)\to (\g_E,1)$ is a graph map. Since finite trees are contractible by Example \ref{eg:acyclic}, Proposition \ref{prop:contriv} gives a based homotopy $h$ from $g$ to the constant map $1$. Every vertex of $\g_E$ corresponds to a path in $\g$ from $v_0$ to some vertex; denote this vertex by $\pi(v)$. Clearly $\pi:(\g_E,1)\to (\g,v_0)$ is a graph map and $\pi\circ g = f$, so the composition $\pi\circ h$ is a based homotopy from $f$ to the constant map $v_0$. Hence $[f]$ is the identity element in $A_n(G)$, proving that $A_n(G)$ is trivial.

We now prove that $g$ is well defined. Let $P=(p_0,\ldots,p_\ell)$ be a path in $\Z^n$. Suppose for some $0<j<\ell$ that the vertices $p_{j-1}$, $p_j$ and $p_{j+1}$ are three corners of a square in $\Z^n$. Let $q_j$ denote the fourth corner of this square. Consider the operation of replacing the term $p_j$ in $P$ with $q_j$. We call this operation a \emph{corner swap}. Next suppose that $p_{k-1}=p_{k+1}$ for some $0<k<\ell$. Consider the operation of removing the terms $p_{k-1}$ and $p_k$ from $P$. We call this operation a \emph{backtrack deletion}.

We show that $\tau(P)$ is invariant under corner swaps and backtrack deletions. Let $p_j$ and $q_j$ be as above. The vertices $p_{j-1}$, $p_j$, $p_{j+1}$ and $q_j$ form a square in $\Z^n$. Since $f$ is a graph map and $G$ contains no squares, the image of this square under $f$ is a path in $G$ consisting of at most 3 vertices. The nontrivial possibilities are illustrated on the right side of Figure \ref{fig:sqtravel} up to relabeling. If the top case holds, then by using the definition of $\tau$, we see that
\begin{align*}
\tau(p_{j-1},p_j,p_{j+1}) &= \tau(p_j,p_{j+1})= \tau(p_j,q_j,p_{j+1})=\tau(p_{j-1},q_j,p_{j+1}).
\end{align*}
If the middle case holds, then
\[\tau(p_{j-1},p_j,p_{j+1})=1=\tau(p_{j-1},q_j,p_{j+1}).\]
If the bottom case holds, then $\tau(p_{j-1},p_j,p_{j+1})=\tau(p_{j-1},q_j,p_{j+1})$ immediately. Hence in any case, we have
\begin{align*}
\tau(P)&=\tau(p_0,\ldots,p_{j-1})\tau(p_{j-1},p_j,p_{j+1})\tau(p_{j+1},\ldots,p_\ell)\\
&=\tau(p_0,\ldots,p_{j-1})\tau(p_{j-1},q_j,p_{j+1})\tau(p_{j+1},\ldots,p_\ell)\\
&=\tau(p_0,\ldots,p_{j-1},q_j,p_{j+1},\ldots,p_\ell),
\end{align*}
proving that $\tau(P)$ is invariant under corner swaps. To prove that $\tau(P)$ is invariant under backtrack deletions, let $p_k$ be as in the previous paragraph, so that $p_{k-1}=p_{k+1}$. We have $\tau(p_{k-1},p_k,p_{k+1})=1$, so
\begin{align*}
\tau(P)&=\tau(p_0,\ldots,p_{k-1})\tau(p_{k-1},p_k,p_{k+1})\tau(p_{k+1},\ldots,p_\ell)\\
&=\tau(p_0,\ldots,p_{k-1})\tau(p_{k+1},\ldots,p_\ell)\\
&=\tau(p_0,\ldots,p_{k-2},p_{k+1})\tau(p_{k+1},\ldots,p_\ell)\\
&=\tau(p_0,\ldots,p_{k-2},p_{k+1},\ldots,p_\ell),
\end{align*}
as desired.

\begin{figure}[ht]
\begin{tikzpicture}[scale=3]
\draw (0,0) node[fill=white] {$q_j$} -- (1,0) node[fill=white] {$p_{j-1}$} -- (1,1) node[fill=white] {$p_j$} -- (0,1) node[fill=white] {$p_{j+1}$} --cycle;

\def\a{3}
\def\b{0.4}
\def\c{0.6}
\draw (\a-\b,5/6) node[left,fill=white] {$f(q_j)=f(p_{j+1})$} -- (\a+\b,5/6) node[right,fill=white] {$f(p_{j-1})=f(p_j)$};
\draw (\a-\b-\c,1/2) node[left,fill=white] {$f(p_{j+1})$} -- (\a,1/2) node[fill=white] {$f(p_j)=f(q_j)$} -- (\a+\b+\c,1/2) node[right,fill=white] {$f(p_{j-1})$};
\draw (\a-\b-\c,1/6) node[left,fill=white] {$f(q_j)$} -- (\a,1/6) node[fill=white] {$f(p_{j-1})=f(p_{j+1})$} -- (\a+\b+\c,1/6) node[right,fill=white] {$f(p_j)$};
\end{tikzpicture}
\caption{A square in $\Z^n$ and three possibilities for its image under $f$.}
\label{fig:sqtravel}
\end{figure}
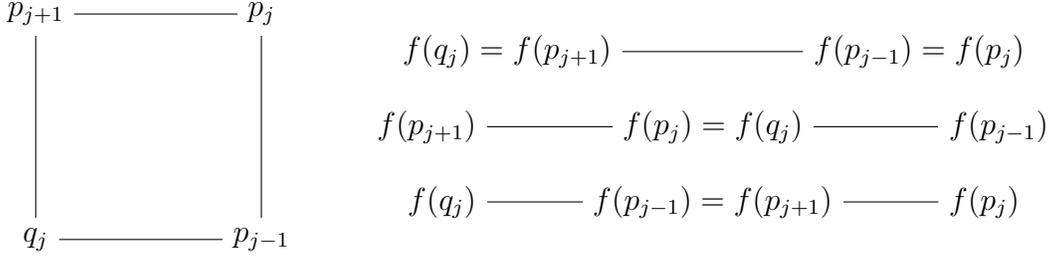

If $P=(p_0,\ldots,p_\ell)$ is a closed path in $\Z^n$, then it is not hard to see that $P$ can be transformed into the path $(p_0,p_1,p_0)$ by a sequence of corner swaps and backtrack deletions. Since $\tau(P)$ is invariant under these operations, we must have
\[\tau(P)=\tau(p_0,p_1,p_0)=1\]
whenever $P$ is closed. Let $x\in \Z^n$, and let $p,q\in \znger$. Let $P$ be a path in $\Z^n$ from $p$ to $x$ and $Q$ a path from $q$ to $x$. Let $R$ be a path from $q$ to $p$ with all vertices in $\znger$. Since $f$ is constant on $\znger$, we have $\tau(R)=1$. Write $PQ^{-1}R$ for the path obtained by concatenating $P$, the reverse of $Q$ and $R$. Since $PQ^{-1}R$ is closed, we have $\tau(PQ^{-1}R)=1$. Thus
\[1=\tau(PQ^{-1}R)=\tau(P)\tau(Q)^{-1}\tau(R)=\tau(P)\tau(Q)^{-1},\]
giving $\tau(P)=\tau(Q)$. It follows that if a path begins in $\znger$, its value under $\tau$ depends only on its endpoint. In other words, the function $g$ from \eqref{eq:gfunc} is well defined.
\end{proof}
\section{Cones and suspensions}
\label{sec:conesus}

We define discrete analogs of the cone and suspension functors from classical topology, and we classify the contractible cones and suspensions of graphs. This material plays a crucial role in the proof of Theorem \ref{thm:mainthm2}.

\begin{mydef}
For $s\geq 1$, let $C_s\g$ denote the graph obtained from $\g\times I_s$ by contracting the subgraph $\g\times \{0\}$. We call $C_s\g$ a \emph{cone} on $\g$.
\end{mydef}

\begin{prop}
	The cone $C_s\g$ on $\g$ is contractible for any $s\geq 1$.
\begin{proof}
Write $C\g=C_s\g$. For each $v\in \g$ and $i=1,\ldots,s$, there is a corresponding vertex $v_i$ of $C\g$ contained in $\g\times\{i\}$. We let $v_0$ denote the lone vertex in the image of $\g\times\{0\}$ in $C\g$. Thus the vertices of $C\g$ are indexed by $v_i$ for $v\in\g$ and $i=0,\ldots,n$, with the understanding that $v_0=w_0$ for all vertices $v$ and $w$ of $\g$. This labeling scheme is illustrated in Figure \ref{fig:cone}. Let $h:C\g\times I_s$ be given by
\[c(v_i,j)=\begin{cases}v_{s-j}&\mbox{if }i>s-j\\v_i&\mbox{if }i\leq s-j.\end{cases}\]
It is routine to check that $c$ is a graph map. In addition, $c(-,0)$ is the identity map on $C\g$ and $c(-,s)=v_0$. Hence $c$ is a contraction of $C\g$.
\end{proof}
\label{prop:conecon}
\end{prop}

\begin{figure}[ht]
\begin{tikzpicture}[scale=2]
\coordinate (a) at (0,0);
\coordinate (b) at (0.5,0.8);
\coordinate (c) at (1,0.5);
\coordinate (d) at (1,-0.5);
\coordinate (e) at (0.5,-0.8);
\draw[rounded corners=3pt] (a)node[left] {$a$}--(b) node[left] {$b$}--(c) node[right] {$c$}--(d) node[right] {$d$}--(e) node[left] {$e$}--cycle;
\draw[fill=black] (a) circle (0.8pt);
\draw[fill=black] (b) circle (0.8pt);
\draw[fill=black] (c) circle (0.8pt);
\draw[fill=black] (d) circle (0.8pt);
\draw[fill=black] (e) circle (0.8pt);

\def\s{0.8}
\coordinate (o) at (2.2,0);
\coordinate (o2) at (2.4,0);
\coordinate (a1) at ($(o)+(a)+(\s,0)$);
\coordinate (b1) at ($(o)+(b)+(\s,0)$);
\coordinate (c1) at ($(o)+(c)+(\s,0)$);
\coordinate (d1) at ($(o)+(d)+(\s,0)$);
\coordinate (e1) at ($(o)+(e)+(\s,0)$);
\coordinate (a2) at ($(o)+(a)+(3*\s,0)$);
\coordinate (b2) at ($(o)+(b)+(3*\s,0)$);
\coordinate (c2) at ($(o)+(c)+(3*\s,0)$);
\coordinate (d2) at ($(o)+(d)+(3*\s,0)$);
\coordinate (e2) at ($(o)+(e)+(3*\s,0)$);
\draw (o2)--(a1)--(a2);
\draw (o2)--(b1)--(b2);
\draw (o2)--(c1)--(c2);
\draw (o2)--(d1)--(d2);
\draw (o2)--(e1)--(e2);
\draw[fill=black] (o2) circle (0.8pt);
\draw[fill=black] (a1) circle (0.8pt);
\draw[fill=black] (b1) circle (0.8pt);
\draw[fill=black] (c1) circle (0.8pt);
\draw[fill=black] (d1) circle (0.8pt);
\draw[fill=black] (e1) circle (0.8pt);
\draw[fill=black] (a2) circle (0.8pt);
\draw[fill=black] (b2) circle (0.8pt);
\draw[fill=black] (c2) circle (0.8pt);
\draw[fill=black] (d2) circle (0.8pt);
\draw[fill=black] (e2) circle (0.8pt);
\draw (o2) node[left] {$a_0$};
\draw (a1)node[label={[label distance=1mm]350:$a_1$}] {}--(b1) node[above left] {$b_1$}--(c1) node[above right] {$c_1$}--(d1) node[below right] {$d_1$}--(e1) node[below left] {$e_1$}--cycle;
\draw (a2)node[right] {$a_2$}--(b2) node[above right] {$b_2$}--(c2) node[right] {$c_2$}--(d2) node[right] {$d_2$}--(e2) node[below right] {$e_2$}--cycle;
\end{tikzpicture}
\caption{The pentagon $\g=\Z_5$, left; and the cone $C_2\g$, right, with labeling scheme as in the proof of Proposition \ref{prop:conecon}.}
\label{fig:cone}
\end{figure}
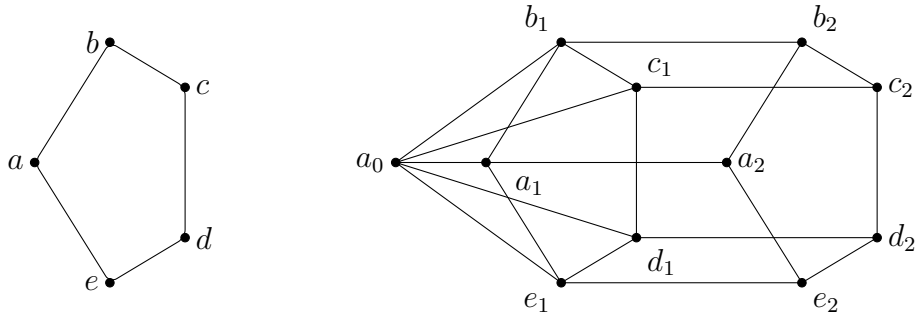

\begin{mydef}
For $t\geq 2$, let $S_t\g$ be the graph obtained from $\g\times I_t$ by contracting $\g\times \{0\}$ and $\g\times \{t\}$ to single vertices. We call $S_t\g$ a \emph{suspension} of $\g$. 
\label{def:sus}
\end{mydef}

\begin{prop}
If $t=2$ or $\g$ is contractible, then $S_tG$ is contractible.

\label{prop:suscon}
\begin{proof}
We will need a labeling scheme for the vertices of $S_t\g$. For each $v\in \g$ and $i=0,\ldots,t$, we let $v_i$ denote the corresponding vertex in the image of $\g\times \{i\}$ in $S\g$. Thus, for example, $u_0=v_0$ and $u_t=v_t$ for all $u,v\in G$. This labeling scheme is illustrated in Figure \ref{fig:susp}, where the pentagon $\g=\Z_5$ is labeled as in Figure \ref{fig:cone}.

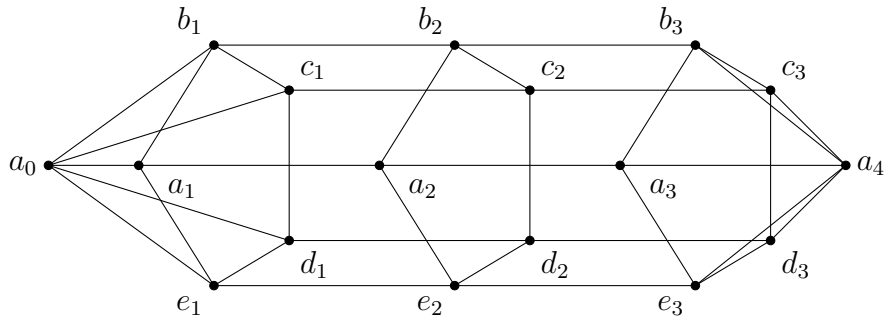
\begin{figure}[ht]
\centering
\begin{tikzpicture}[scale=2]
\coordinate (a) at (0,0);
\coordinate (b) at (0.5,0.8);
\coordinate (c) at (1,0.5);
\coordinate (d) at (1,-0.5);
\coordinate (e) at (0.5,-0.8);
\def\s{0.8}
\coordinate (o) at (0,0);
\coordinate (o2) at (0.2,0);
\coordinate (a1) at ($(o)+(a)+(\s,0)$);
\coordinate (b1) at ($(o)+(b)+(\s,0)$);
\coordinate (c1) at ($(o)+(c)+(\s,0)$);
\coordinate (d1) at ($(o)+(d)+(\s,0)$);
\coordinate (e1) at ($(o)+(e)+(\s,0)$);
\coordinate (a2) at ($(o)+(a)+(3*\s,0)$);
\coordinate (b2) at ($(o)+(b)+(3*\s,0)$);
\coordinate (c2) at ($(o)+(c)+(3*\s,0)$);
\coordinate (d2) at ($(o)+(d)+(3*\s,0)$);
\coordinate (e2) at ($(o)+(e)+(3*\s,0)$);
\coordinate (a3) at ($(o)+(a)+(5*\s,0)$);
\coordinate (b3) at ($(o)+(b)+(5*\s,0)$);
\coordinate (c3) at ($(o)+(c)+(5*\s,0)$);
\coordinate (d3) at ($(o)+(d)+(5*\s,0)$);
\coordinate (e3) at ($(o)+(e)+(5*\s,0)$);
\coordinate (p) at ($(o)+(5.5,0)$);
\draw (o2)--(a1)--(a2)--(a3)--(p);
\draw (o2)--(b1)--(b2)--(b3)--(p);
\draw (o2)--(c1)--(c2)--(c3)--(p);
\draw (o2)--(d1)--(d2)--(d3)--(p);
\draw (o2)--(e1)--(e2)--(e3)--(p);
\draw[fill=black] (o2) circle (0.8pt);
\draw[fill=black] (a1) circle (0.8pt);
\draw[fill=black] (b1) circle (0.8pt);
\draw[fill=black] (c1) circle (0.8pt);
\draw[fill=black] (d1) circle (0.8pt);
\draw[fill=black] (e1) circle (0.8pt);
\draw[fill=black] (a2) circle (0.8pt);
\draw[fill=black] (b2) circle (0.8pt);
\draw[fill=black] (c2) circle (0.8pt);
\draw[fill=black] (d2) circle (0.8pt);
\draw[fill=black] (e2) circle (0.8pt);
\draw[fill=black] (a3) circle (0.8pt);
\draw[fill=black] (b3) circle (0.8pt);
\draw[fill=black] (c3) circle (0.8pt);
\draw[fill=black] (d3) circle (0.8pt);
\draw[fill=black] (e3) circle (0.8pt);
\draw[fill=black] (p) circle (0.8pt);
\draw (o2) node[left] {$a_0$};
\draw (p) node[right] {$a_4$};
\draw (a1)node[label={[label distance=1mm]350:$a_1$}] {}--(b1) node[above left] {$b_1$}--(c1) node[above right] {$c_1$}--(d1) node[below right] {$d_1$}--(e1) node[below left] {$e_1$}--cycle;
\draw (a2)node[label={[label distance=1mm]350:$a_2$}] {}--(b2) node[above left] {$b_2$}--(c2) node[above right] {$c_2$}--(d2) node[below right] {$d_2$}--(e2) node[below left] {$e_2$}--cycle;
\draw (a3)node[label={[label distance=1mm]350:$a_3$}] {}--(b3) node[above left] {$b_3$}--(c3) node[above right] {$c_3$}--(d3) node[below right] {$d_3$}--(e3) node[below left] {$e_3$}--cycle;
\end{tikzpicture}
\caption{The labeling scheme of $S_4 \Z_5$.}
\label{fig:susp}
\end{figure}

Suppose that $t=2$, and fix $w\in G$. There is contraction $c:S_t\g\times I_2\to S_t\g$ of $S_tG$ to $v_0$ given by
\[c(v_i,1)=\begin{cases} v_0&\mbox{if }i<2\\w_1&\mbox{if }i=2.\end{cases}\]
It is routine to verify that $c$ is a graph map.

Suppose now that $t$ is arbitrary and $\g$ is contractible. Let $C:\g\times I_m\to \g$ be a contraction of $G$ to $w\in G$. Define a function $\kappa:S_t\g\times I_{m+t}\to S_t\g$ as follows:
\[\kappa(v_i,j)=\begin{cases} C(v,j)_i&\mbox{if }0\leq j\leq m\\ w_{i+m-j}&\mbox{if }m<j\leq i+m\\ v_0&\mbox{if }i+m<j\leq m+t.\end{cases}\]
It is routine to verify that $C$ is a contraction of $S_tG$ to $v_0$.
\end{proof}
\end{prop}

\section{Toward a discrete Hurewicz theorem}
\label{sec:hurewicz}

We prove Theorem \ref{thm:hursurj} below, which implies Theorem \ref{thm:mainthm2} from the introduction. We first recall the definition of discrete singular cubical homology groups of graphs and define the map $\psi:A_n(G)\to \hc_n(G)$ from \eqref{eq:intropsi}.

\subsection{Discrete singular cubical homology of graphs}

Let $I_m^n$ denote the $n$-fold Cartesian prouct of $I_m$, regarded as a subgraph of $\Z^n$. For each $n\geq 0$, let $Q_n$ be the graph given by
\[
Q_n=\begin{cases}I_0&\mbox{if }n=0\\ I_1^n&\mbox{if }n\geq 1.\end{cases}\]
An \emph{$n$-cube of $G$} is a graph map $\sigma :Q_n\to G$. For $1\leq i\leq n$, let $D_i^-\sigma$ and $D_i^+\sigma$ be the $(n-1)$-cubes of $G$ given by
\begin{align*}
D_i^-\sigma(a_1,\ldots,a_{n-1}) &= \sigma(a_1,\ldots,a_{i-1},0,a_i,\ldots,a_{n-1})\\
D_i^+\sigma(a_1,\ldots,a_{n-1}) &= \sigma(a_1,\ldots,a_{i-1},1,a_i,\ldots,a_{n-1}).
\end{align*}
We say that $\sigma$ is \emph{degenerate} if $D_i^-\sigma = D_i^+\sigma$ for some $i$. Let $L_n(G)$ denote the free $\Z$-module generated by all $n$-cubes of $G$, and let $D_n(G)$ denote the submodule generated by the degenerate $n$-cubes of $G$. Let
\[C_n(G)=L_n(G)/D_n(G).\]
The elements of $C_n(G)$ are called \emph{$n$-chains} of $G$.


Given an $n$-cube $\sigma$ of $G$ with $n\geq 1$, let
\[\partial_n(\sigma) = \sum_{i=1}^n (-1)^i (D_i^- \sigma - D_i^+ \sigma).\]
Extend linearly to obtain a map $\partial_n:C_n(G)\to C_{n-1}(G)$. It is routine to check that $(C_\bullet(G),\partial_\bullet)$ is a chain complex. Let
\[\hc_n(G)=\ker(\partial_n)/\operatorname{im}(\partial_{n+1}).\]
The group $\hc_n(G)$ is called the \emph{$n$th discrete singular cubical homology group} of $G$. If $z\in \ker(\partial_n)$, then we write $\overline{z}$ for the equivalence class of $z$ in $\hc_n(G)$.

There is a discrete Hurewicz theorem in dimension 1:

\begin{thm}[{\cite[Theorem 4.1]{barcelo2014}}]
For any graph $G$, there is a surjective map
\begin{equation}
\psi: A_1(G)\to \hc_1(G)
\label{eq:hurmap1}
\end{equation}
whose kernel is $[A_1(G),A_1(G)]$.
\label{thm:hurmap1}
\end{thm}

The map $\psi$ can be described explicitly. For this, and for the remainder of this section, we will represent based homotopy classes $[f]\in A_n(G)$ as graph maps
\[(\imn,\bimn)\to (G,v_0).\]
This is justified by the following argument. Let $f:(\Z^n,\partial \Z^n)\to (G,v_0)$ be any graph map of radius $r$. Let $h:\Z^n\times I_r\to G$ be given by $h(x,i) = f(x_1-i,\ldots,x_n-i)$ for all $(x,i)\in \Z^n\times I_r$. It is routine to check that $h$ is a based homotopy from $h_0=f$ to a graph map $h_r$ that is constant outside the interior of $I_m^n$. By restricting our attention to the domain $I_m^n$, we obtain a map $(\imn,\bimn)\to (G,v_0)$ that is based homotopic to $f$.

Following the above, suppose that $f:(I_m,\partial I_m)\to (G,v_0)$ is a graph map, and for $i=0,\ldots, m-1$ let $f^i:Q_1\to G$ be given by $f^i(0)=f(i)$ and $f^i(1)=f(i+1)$. We have
\[\psi([f])=\sum_{i=0}^{m-1} f^i.\]
For higher dimensions, we conjecture a direct analog of the classical theorem.

\begin{conj}
If $n\geq 2$ and $A_k(G)$ is trivial for all $k<n$, then $A_n(G)\cong \hc_n(G)$.

\label{conj:hur}
\end{conj}

\begin{eg}
Let $G_1$ be a graph containing no 3- or 4-cycles. It follows from Proposition \ref{prop:cwcomp} and \cite[Proposition 1A.2]{hatcher2002} that $A_1(G_1)$ is a free group of rank $k$, where $k$ is the number of edges of $G_1$ not contained in a given spanning tree. Theorem \ref{thm:hurmap1} implies that
\[\hc_1(G_1)\cong \Z^k,\]
where $\Z^k$ denotes a free abelian group of rank $k$. An explicit set of generators, although not always minimal one, is the set of $\psi([f_Z])$ for cycles $Z$ of $G_1$, where the maps $f_Z$ are defined as in Example \ref{eg:babyfg}. For each $n\geq 1$, let $G_{n+1}=S_{n+3}G_n$. It will follow from Proposition \ref{prop:homgen}(ii) below that $\hc_n(G_n)\cong \Z^k$ for all $n\geq 1$. This result originally appeared in \cite{barcelo2019}.
\label{eg:homsus}
\end{eg}

\subsection{A proposed discrete Hurewicz map}

We define a map
\[\psi: A_n(G)\to \hc_n(G)\]
that generalizes \eqref{eq:hurmap1}. Let $I_{m_1,\ldots,m_n}\subseteq \Z^n$ be the grid graph given by
\begin{equation}
I_{m_1,\ldots,m_n} = I_{m_1} \times \cdots \times I_{m_n}.
\end{equation}
For the following definitions, let $L=I_{m_1,\ldots,m_n}$ and $L'=I_{m_1-1,\ldots,m_n-1}$ for some $m_i\geq 1$, and let $f:L\to G$ be a graph map.

\begin{mydef}
For each $x\in L'$, let $f^x : Q_n\to \g$ be the graph map given by 
\begin{equation}
f^x(y)= f(x+y).
\end{equation}
The maps $f^x$ are called the \emph{$n$-cubes of $f$}. If $x+Q_n\subseteq V$ for some $V\subseteq \Z^n$, then we say that $f^x$ is \emph{contained in $V$}.
\end{mydef}

\begin{mydef}
Define an element $\phi(f)$ of $C_n(G)$ by
\begin{equation}
\phi(f)=\sum_{x\in L'} f^x.
\end{equation}
In other words, $\phi(f)$ is the sum of all $n$-cubes of $f$.
\end{mydef}

\begin{lemma}
We have the following:
\begin{enumerate}[(i)]
\item If $[f]\in A_n(G)$, then $\phi(f)\in \ker(\partial_n)$
\item If $[f]=[g]$ in $A_n(G)$, then $\overline{\phi(f)}=\overline{\phi(g)}$.
\end{enumerate}
\begin{proof}
First we prove (i). Let $f:(\imn,\bimn)\to (G,v_0)$ be a graph map. For $i=1,\ldots,n$ let $\Omega_i=\{x\in I_{m-1}^n:x_i=0\}$. Let $\epsilon_i\in \Z^n$ be the $i$th standard basis vector. We have
\[\partial_n \phi(f) = \sum_{i=1}^n (-1)^i \sum_{x\in \Omega_i} \sum_{r=0}^{m-1} D^-_i f^{x+r\epsilon_i}-D^+_i f^{x+r\epsilon_i}.\]
For each $i$ and each $x\in \Omega_i$, the innermost sum telescopes, and the leftover portion
\[D^-_i f^x - D^+_i f^{x+(m-1)\epsilon_i}\]
is 0, since $f$ is constant on $\bimn$. This proves (i).

To prove (ii), let $h$ be a homotopy from $f$ to $g$, say $h:I_m^{n+1}\to G$. By another telescoping argument, it is not hard to see that
\[\partial_{n+1} \phi(h) = (-1)^{n+1}( \phi(f) - \phi(g))\]
in $C_n(G)$, proving (ii).
\end{proof}
\label{lem:hurmap}
\end{lemma}


\begin{mydef}
Let $\psi:A_n(G)\to \hc_n(G)$ be the function given by $\psi([f])=\overline{\phi(f)}$.
\end{mydef}

\begin{prop}
The function $\psi:A_n(G)\to \hc_n(G)$ is a group homomorphism.
\begin{proof}
If $[f]$ and $[g]$ are any two elements of $A_n(G)$, and $p$ is defined as in \eqref{eq:concat}, then $\phi(p)-(\phi(f)+\phi(g))$ is a sum of constant $n$-cubes $Q^n\to v_0$, which are degenerate. Hence $\psi([f]\cdot [g])=\psi([p])=\psi([f])+\psi([g])$ in $\hc_n(G)$, as desired.
\end{proof}

\begin{remark}
If we regard $A_n$ (resp., $\hc_n$) as a functor from the category of connected simple graphs with graph maps to the category of groups (resp., abelian groups), then $\psi$ is a natural transformation from $A_n$ to $\hc_n$.
\end{remark}

\label{prop:groupmap}
\end{prop}

\subsection{Surjectivity in a special case}
\label{sec:surj}

Given any $n\geq 1$, we describe an infinite family of graphs $G$ for which the map $\psi:A_n(G)\to \hc_n(G)$ is surjective. From this we obtain an infinite class of $G$ for which $A_n(G)$ is nontrivial. Throughout the section, we let $G_1$ be a graph and
\begin{equation}
G_{n+1}=S_{n+3}G_n
\label{eq:gndef}
\end{equation}
for all $n\geq 1$, where $S_{n+3}$ is the suspension functor from Definition \ref{def:sus}. We think of $G_{n+1}$ as consisting of $n+2$ copies of $G_n$ and two extra vertices, called the \emph{north pole} and \emph{south pole}. We now state the main result of this section, which implies Theorem \ref{thm:mainthm2}.

\begin{thm}
If $G_1$ contains no 3- or 4-cycles, then the map
\[\psi :A_k(G_n)\to \hc_k(G_n)\]
is surjective for all $1\leq k\leq n$.
\label{thm:hursurj}
\end{thm}

\begin{cor}
If $G_1$ contains a cycle but does not contain any 3- or 4-cycles, then $A_n(G_n)$ is infinite for all $n\geq 1$.
\label{cor:hursurj}

\begin{proof}
This follows from Example \ref{eg:homsus}, Theorem \ref{thm:hursurj} and  Proposition \ref{prop:homgen}(ii) below.
\end{proof}
\end{cor}

The proof of Theorem \ref{thm:hursurj} will proceed in several steps. We will focus on the case where $G_1$ is a cycle graph of length 5, i.e. $G_1=\Z_5$, and then extend the proof to the general case. Example \ref{eg:homsus} says that if $G_1=\Z_5$, then $\hc_n(G_n)\cong \Z$ for all $n\geq 1$. The following proposition, which applies to any choice of $G_1$, will allow us to identify an explicit generator of $\hc_n(G_n)$ when $G_1=\Z_5$. Part (ii) first appeared as \cite[Theorem 5.2]{barcelo2019}.

\begin{prop}
If $n\geq 2$, then
\begin{enumerate}[(i)]
\item $\hc_k(G_n)=0$ for $k=1,\ldots,n-1$
\item There is an explicit isomorphism $\Delta:\hc_n(G_n)\to \hc_{n-1}(G_{n-1})$.
\end{enumerate}

\begin{proof}
We first prove (ii), adapting our argument from the proof of \cite[Theorem 5.2]{barcelo2019}. Let $n\geq 2$. Let $X$ (resp., $Y$) be the graph obtained from $G_n$ by deleting the north (resp., south) pole. From the proof of \cite[Theorem 5.2]{barcelo2019} we have the following exact sequences for $k=1,\ldots,n$:
\begin{equation}
\hc_k(X)\oplus \hc_k(Y) \to \hc_k(G_n)\xrightarrow{\,\,\partial_*\,\,} \hc_{k-1}(X\cap Y) \to \hc_{k-1}(X)\oplus \hc_{k-1}(Y).
\label{eq:ses}
\end{equation}
The map $\partial_*$ can be described explicitly: if $\overline{z}\in \hc_k(G_n)$, then $z=x+y$ for some $x\in C_k(X)$ and $y\in C_k(Y)$, and $\partial_*(\overline{z}) = \overline{\partial_k(x)}$. Note that $G_{n-1}\times \{1\}$ is a deformation retract of $X\cap Y$ in the sense of \cite[Section 4]{barcelo2019-1}. Hence if we identify $G_{n-1}$ with $G_{n-1}\times \{1\}$, then the inclusion map $\iota: G_{n-1}\to X\cap Y$ induces an isomorphism $\iota_*:\hc_{n-1}(G_{n-1})\to \hc_{n-1}(X\cap Y)$. Let
\[\Delta:=(\iota_*)^{-1}\circ \partial_*.\]
Proposition \ref{prop:conecon} implies that $X$ and $Y$ are contractible, so $\hc_k(X)=\hc_k(Y)=0$ for all $k\geq 1$. The exact sequences \eqref{eq:ses} now become
\[0\to \hc_k(G_n)\xrightarrow{\,\,\Delta\,\,} \hc_{k-1}(G_{n-1}) \to \hc_{k-1}(X)\oplus \hc_{k-1}(Y).\]
For $k\geq 2$ we have $\hc_{k-1}(X)\oplus \hc_{k-1}(Y)=0$, so $\Delta$ is an isomorphism, proving (ii).

We now prove (i). If $\g=(V,E)$ is any connected graph, then $\hc_0(G) \cong \Z$ has presentation
\begin{equation}
\hc_0(G)=\langle v\in V\mid u=v \mbox{ for all } u,v\in V\rangle.
\label{eq:hc0}
\end{equation}
Any $x\in C_1(X)$ can be written as $x=\sum_{i=1}^\ell m_i f_i$ for integers $m_i$ and graph maps $f_i:Q_1\to X$. We have $\partial_1(x) = \sum_{i=1}^\ell m_i(f_i(1)-f_i(0))$, so $\overline{\partial_1(x)} = \sum_{i=1}^\ell m_i(f_i(1)-f_i(1)) = 0$ in $\hc_0(G_{n-1})$ by \eqref{eq:hc0}. Thus when $k=1$, we have $\operatorname{im}\Delta=0$. But also $\ker\Delta=0$ by exactness, so $\hc_1(G_n)=0$ for all $n\geq 2$. Hence if $k<n$, then by applying $(i)$ repeatedly we obtain an isomorphism $\hc_k(G_n)\to \hc_1 (G_{n-k+1})=0$, proving (i).
\end{proof}
\label{prop:homgen}
\end{prop}

We will need a labeling scheme for the vertices of $G_n$. The graph $G_{n+1}$ consists of $n+2$ isomorphic copies of $G_n$ suspended between the north and south poles. For each $i=0,\ldots,n+3$ we obtain a graph map $\iota_n^i : G_n\to G_{n+1}$ taking each $v\in G_n$ to the corresponding vertex in the $i$th copy of $G_n$. Thus $\iota_n^0$ and $\iota_n^{n+3}$ are the constant graph maps taking $G_n$ to the north and south poles of $G_{n+1}$, respectively. For $v\in G_n$, we write $v_i = \iota_n^i(v)$ for all $i$, and we say that $v_i$ is obtained from $v$ by \emph{adding the subscript} $i$. In general, if $v$ is a vertex of $G_k$ for $k< n$, then we write
\begin{equation}
v_{i_k,\ldots,i_n} = \iota_n^{i_n}\cdots \iota_k^{i_k} v.
\label{eq:addsub2}
\end{equation}
For example, $v_{i_k,\ldots,i_n}$ is the south pole of $G_{n+1}$ if and only if $i_n=0$.

If $G$ is a graph and $L=I_{m_1,\ldots,m_n}$ is a grid graph, then we will think of a map $L\to G$ as an $n$-dimensional array of size $(m_1+1)\times\cdots\times(m_n+1)$, whose entries are vertices of $G$. For example, when $n=1$, a map $L\to G$ is a single column with $m_1+1$ entries. When $n=2$, we get an $(m_1+1)\times (m_2+1)$ matrix. In both cases, we consider the top-left entry to be the image of the origin $(0,\ldots,0)\in L$.

Suppose that $G_1=\Z_5$. Our task now is to construct a generator of $\hc_n(G_n)\cong \Z$ inductively. Define grid graphs $J_n$ by $J_1=I_5$ and
\begin{equation}
J_{n+1}=J_n\times I_{n+3}
\label{eq:jn}
\end{equation}
for all $n\geq 1$. Label the vertices of $G_1$ as in Figure \ref{fig:cone}. As just described, we will think of maps $J_n\to G$ as $n$-dimensional arrays of appropriate size. Let $\gamma_1:J_1\to G_1$ be given by
\begin{equation}
\gamma_1 = \begin{pmatrix} a & b & c & d & e & a\end{pmatrix}^T,
\end{equation}
where $T$ denotes the usual transpose. For $n\geq 1$ let $\gamma_{n+1}:J_{n+1}\to G_{n+1}$ be given by
\begin{equation}
\gamma_{n+1}(v,i) = \gamma_n(v)_i
\end{equation}
for all $(v,i)\in J_n\times I_{n+3}$, where we have added the subscript $i$ to $\gamma_n(v)$ as defined in \eqref{eq:addsub2}. For example, $\gamma_2:I_5\times I_4\to G$ is given by
\[
\gamma_2=\begin{pmatrix}
a_0 & a_0 & a_0 & a_0 & a_0 & a_0 \\
a_1 & b_1 & c_1 & d_1 & e_1 & a_1 \\
a_2 & b_2 & c_2 & d_2 & e_2 & a_2 \\
a_3 & b_3 & c_3 & d_3 & e_3 & a_3 \\
a_4 & a_4 & a_4 & a_4 & a_4 & a_4 
\end{pmatrix}^T.\]
It is routine to verify that $\gamma_n:J_n\to G_n$ is a graph map for all $n$.

We record some basic properties of $\gamma_n$. First, since $g_1(0)=g_1(5)$, we have
\begin{equation}
\gamma_n(0,v_2,\ldots,v_n)=\gamma_n(5,v_2,\ldots,v_n)
\label{eq:gflat05}
\end{equation}
for all $v\in J_n$. We also have
\begin{equation}
\gamma_n(v_1,\ldots,v_{i-1},0,v_{i+1},\ldots,v_n) = \gamma_n(0,\ldots,0,v_{i+1},\ldots,v_n)
\label{eq:gflat}
\end{equation}
for all $i=2,\ldots,n$, since adding the subscript 0 to any vertex of $G_{i-1}$ gives the south pole of $G_i$. Similarly, adding the subscript $i+2$ to any vertex of $G_{i-1}$ gives the north pole of $G_i$, so
\begin{equation}
\gamma_n(v_1,\ldots,v_{i-1},i+2,v_{i+1},\ldots,v_n) = \gamma_n(0,\ldots,0,i+2,v_{i+1},\ldots,v_n).
\label{eq:gflat2}
\end{equation}
Finally, note that $\partial_n \phi(\gamma_n) = 0$ by a telescoping argument, so we can consider the homology class $\overline{\phi(\gamma_n)}$. 

\begin{lemma}
If $G_1=\Z_5$ and $\gamma_n$ is defined as above for $n\geq 1$, then
\begin{enumerate}[(i)]
\item $\hc_n(G_n)=\langle \overline{\phi(\gamma_n)}\rangle$
\item $\Delta(\overline{\phi(\gamma_{n+1})}) = (-1)^n\overline{\phi(\gamma_n)}$, where $\Delta:\hc_{n+1}(G_{n+1})\to \hc_n(G_n)$ is the isomorphism from Proposition \ref{prop:homgen}(ii).
\end{enumerate}

\begin{proof}
We prove (ii). Example \ref{eg:homsus} implies (i) for $n=1$, so the general case of (i) will follow from (ii). Suppose that $n\geq 2$. We retain the definitions of $X$, $Y$, $\partial_*$, $\iota$ and $\iota_*$ from the proof of Proposition \ref{prop:homgen}. Thus $X$ (resp., $Y$) is obtained from $G_n$ by removing the north (resp., south) pole. We can write $\phi(\gamma_n)=x+y$, where $x\in C_n(X)$ and $y\in C_n(Y)$. In particular, if $U=\{u\in J_n : 1\leq u_n\leq n+1\}$, then we can take
\[y=\sum_{u\in U} \gamma_n^u.\]

Let $V=\{v\in J_n : v_n=1\}$ and $W=\{w\in J_n: w_n=n+1\}$. The only terms in the sum $\partial_n(y)$ that do not cancel by telescoping are
\[\partial_n(y)=(-1)^{n}\left(\sum_{v\in V} D_n^- \gamma_n^v - \sum_{w\in W} D_n^+ \gamma_n^w\right).\]
By construction, each $D_n^+ \gamma_n^w$ is the constant map $Q^{n-1}\to a_{n+2,\ldots,n+2}$. This is a degenerate $(n-1)$-cube of $G_n$, so we can ignore it. We are left with the sum
\[\partial_n(y)=(-1)^n \sum_{v\in V} D_n^- \gamma_n^v,\]
but this is precisely $\phi(\iota\circ \gamma_{n-1})$. Hence
\[\partial_*(\overline{\phi(\gamma_n)}) =\overline{\partial_n(x)}
= -\overline{\partial_n(y)}
= (-1)^{n-1} \overline{\phi(\iota\circ \gamma_{n-1})}
 = (-1)^{n-1} (\iota_*(\overline{\phi(\gamma_{n-1})})),\]
from which we deduce that $\Delta(\overline{\phi(\gamma_n)}) = (\iota_*)^{-1}(\partial_*(\overline{\phi(\gamma_n)})) = (-1)^{n-1} \overline{\phi(\gamma_{n-1})}$, proving (ii).
\end{proof}
\label{lem:hursurj1}
\end{lemma}

\begin{lemma}
If $G_1=\Z_5$, then the image of $\psi:A_n(G_n)\to \hc_n(G_n)$ contains $\overline{\phi(\gamma_n)}$.
\label{lem:hursurj2}
\end{lemma}

We illustrate the proof of Lemma \ref{lem:hursurj2} for $n=2$ before arguing the general case. Let $L_2 = I_{13,8}$, and let $M_2=(4,0)+I_{5,4}\subseteq L_2$. Let $f_2:(L_2,\partial L_2)\to (G_2,a_0)$ be the graph map in Figure \ref{fig:c2}. Here $f_2$ is constant on every outlined region.

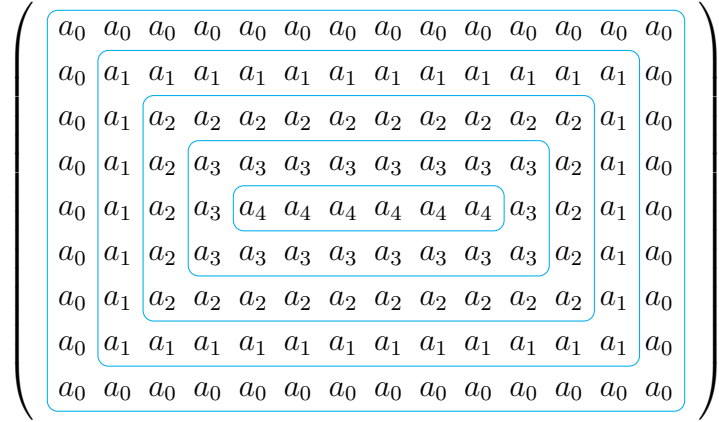
\begin{figure}[ht]
\centering
\begin{tikzpicture}
\matrix[matrix of math nodes, column sep = {6mm,between origins},row sep = {6mm,between origins}, left delimiter = {(}, right delimiter={)}](m){%
a_0 & a_0 & a_0 & a_0 & a_0 & a_0 & a_0 & a_0 & a_0 & a_0 & a_0 & a_0 & a_0 & a_0\\
a_0 & a_1 & a_1 & a_1 & a_1 & a_1 & a_1 & a_1 & a_1 & a_1 & a_1 & a_1 & a_1 & a_0\\
a_0 & a_1 & a_2 & a_2 & a_2 & a_2 & a_2 & a_2 & a_2 & a_2 & a_2 & a_2 & a_1 & a_0\\
a_0 & a_1 & a_2 & a_3 & a_3 & a_3 & a_3 & a_3 & a_3 & a_3 & a_3 & a_2 & a_1 & a_0\\
a_0 & a_1 & a_2 & a_3 & a_4 & a_4 & a_4 & a_4 & a_4 & a_4 & a_3 & a_2 & a_1 & a_0\\
a_0 & a_1 & a_2 & a_3 & a_3 & a_3 & a_3 & a_3 & a_3 & a_3 & a_3 & a_2 & a_1 & a_0\\
a_0 & a_1 & a_2 & a_2 & a_2 & a_2 & a_2 & a_2 & a_2 & a_2 & a_2 & a_2 & a_1 & a_0\\
a_0 & a_1 & a_1 & a_1 & a_1 & a_1 & a_1 & a_1 & a_1 & a_1 & a_1 & a_1 & a_1 & a_0\\
a_0 & a_0 & a_0 & a_0 & a_0 & a_0 & a_0 & a_0 & a_0 & a_0 & a_0 & a_0 & a_0 & a_0\\
};
  \draw[cyan,rounded corners] (m-1-1.north west) rectangle (m-9-14.south east);
  \draw[cyan,rounded corners] (m-1-1.south east) rectangle (m-8-13.south east);
  \draw[cyan,rounded corners] (m-2-2.south east) rectangle (m-7-12.south east);
  \draw[cyan,rounded corners] (m-3-3.south east) rectangle (m-6-11.south east);
  \draw[cyan,rounded corners] (m-4-4.south east) rectangle (m-5-10.south east);
\end{tikzpicture}
\caption{The transpose of $f_2$ with its fibers $B_{2,k}$ outlined in blue.}
\label{fig:c2}
\end{figure}

Let $g_2:(L_2,\partial L_2)\to (G_2,a_0)$ be the graph map in Figure \ref{fig:wrapping}. Here $M_2$ is the region outlined in red. Note that $f_2(x)=g_2(x)$ for all $x\in L_2\setminus M_2^\circ$. Hence if $\Omega_2=\{x\in \Z^2:x+Q_2\subseteq M_2\}$, then
\[\phi(g_2) - \phi(f_2) = \sum_{x\in \Omega_2} g_2^x - \sum_{x\in \Omega_2} f_2^x.\]
The first sum on the right hand side is easily seen to be $\phi(\gamma_2)$ by inspection. Each term in the second sum is a degenerate $2$-cube of $f_2$. Hence $\phi(g_2)-\phi(f_2)=\phi(\gamma_2)$, proving that
\[\psi([g_2]-[f_2])=\overline{\phi(\gamma_2)},\]
as desired.

\begin{figure}[ht]
\centering
\begin{tikzpicture}[scale=0.9]
\matrix[matrix of math nodes, column sep = {6mm,between origins},row sep = {6mm,between origins}, left delimiter = {(}, right delimiter={)}](m){%
{a_0} & a_0 & a_0 & a_0 & a_0 & a_0 & a_0 & a_0 & a_0 & a_0 & a_0 & a_0 & a_0 & {a_0}\\
a_0 & {a_1} & a_1 & a_1 & a_1 & b_1 & c_1 & d_1 & e_1 & a_1 & a_1 & a_1 & {a_1} & a_0\\
a_0 & a_1 & {a_2} & a_2 & a_2 & b_2 & c_2 & d_2 & e_2 & a_2 & a_2 & {a_2} & a_1 & a_0\\
a_0 & a_1 & a_2 & {a_3} & a_3 & b_3 & c_3 & d_3 & e_3 & a_3 & {a_3} & a_2 & a_1 & a_0\\
a_0 & a_1 & a_2 & a_3 & a_4 & a_4 & a_4 & a_4 & a_4 & a_4 & a_3 & a_2 & a_1 & a_0\\
a_0 & a_1 & a_2 & {a_3} & a_3 & a_3 & a_3 & a_3 & a_3 & a_3 & {a_3} & a_2 & a_1 & a_0\\
a_0 & a_1 & {a_2} & a_2 & a_2 & a_2 & a_2 & a_2 & a_2 & a_2 & a_2 & {a_2} & a_1 & a_0\\
a_0 & {a_1} & a_1 & a_1 & a_1 & a_1 & a_1 & a_1 & a_1 & a_1 & a_1 & a_1 & {a_1} & a_0\\
{a_0} & a_0 & a_0 & a_0 & a_0 & a_0 & a_0 & a_0 & a_0 & a_0 & a_0 & a_0 & a_0 & {a_0}\\
};
\def\magentavert{0.137}
\def\r{0.137}
\def\lastvert{0}
  \draw[cyan,rounded corners] (m-3-4.south east) -- (m-3-3.south east) --
    ($(m-7-3.north east)+(0,\r)$) -- ($(m-7-12.north west)+(0,\r)$) -- (m-3-12.south west) -- 
    (m-3-11.south west) -- ($(m-6-11.north west)+(0,\magentavert)$) --
    ($(m-6-4.north east)+(0,\magentavert)$) -- cycle;
  \draw[cyan,rounded corners] (m-2-4.south east) -- (m-2-2.south east) --
    ($(m-8-2.north east)+(0,\r)$) -- ($(m-8-13.north west)+(0,\r)$) -- (m-2-13.south west) -- 
    (m-2-11.south west) -- (m-3-11.south west) -- (m-3-12.south west) -- 
    ($(m-7-12.north west)+(0,\magentavert)$) -- ($(m-7-3.north east)+(0,\magentavert)$) --
    (m-3-3.south east) -- (m-3-4.south east) -- cycle;
  \draw[cyan,rounded corners] (m-1-4.south east) -- (m-1-1.south east) --
    ($(m-9-1.north east)+(0,\r)$) -- ($(m-9-14.north west)+(0,\r)$) -- (m-1-14.south west) -- 
    (m-1-11.south west) -- (m-2-11.south west) -- (m-2-13.south west) -- 
    ($(m-8-13.north west)+(0,\magentavert)$) -- ($(m-8-2.north east)+(0,\magentavert)$) --
    (m-2-2.south east) -- (m-2-4.south east) -- cycle;
  \draw[cyan,rounded corners] ($(m-1-4.north east)+(0,\lastvert)$) -- ($(m-1-1.north west)+(0,\lastvert)$) --
    (m-9-1.south west) -- (m-9-14.south east) -- ($(m-1-14.north east)+(0,\lastvert)$) --
    ($(m-1-11.north west)+(0,\lastvert)$) -- (m-1-11.south west) -- (m-1-14.south west) -- 
    ($(m-9-14.north west)+(0,\magentavert)$) -- ($(m-9-1.north east)+(0,\magentavert)$) --
    (m-1-1.south east) -- (m-1-4.south east) -- cycle;
   \draw[thick,magenta,rounded corners] (m-1-4.north east) rectangle ($(m-6-11.north west)+(0,\r)$);
\end{tikzpicture}
\caption{The transpose of $g_2$ with $M_2$ outlined in red and each $B_{2,k}\setminus M_2$ outlined in blue.}
\label{fig:wrapping}
\end{figure}
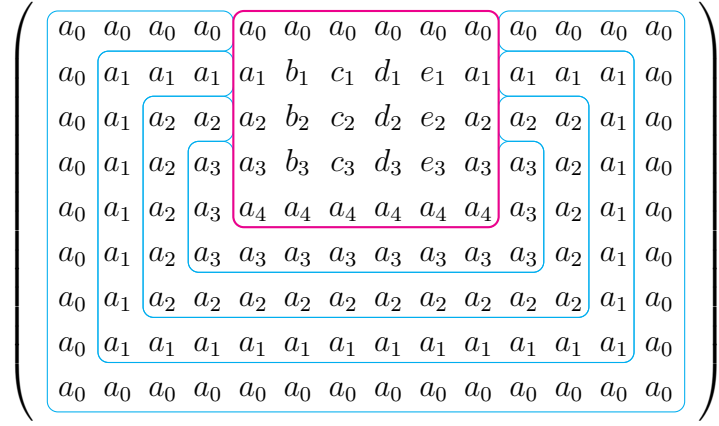

\begin{proof}[Proof of Lemma \ref{lem:hursurj2}]
The case $n=1$ is trivial; simply note that
\[\psi([\gamma_1])=\overline{\phi(\gamma_1)}.\]
Suppose now that $n\geq 2$. We will construct an $n$-dimensinal grid graph $L_n$ and graph maps
\[f_n,g_n:(L_n,\partial L_n)\to (G_n,a_{0,\ldots,0})\]
such that
\begin{equation}
\psi([g_n]-[f_n]) = \overline{\phi(\gamma_n)},
\label{eq:psigf}
\end{equation}
proving the result.

For $n\geq 2$ let $L_n = I_{m_1,\ldots,m_n}$, where
\[m_i = \begin{cases}(n-i+1)(n+i+4) - 1&\mbox{if }i=1\\(n-i+1)(n+i+4)&\mbox{if }i>1.\end{cases} \]
For example, the first several $L_n$ are
\begin{align*}
L_2&=I_{13}\times I_8\\
L_3&=I_{23}\times I_{18}\times I_{10}\\
L_4&=I_{35}\times I_{30}\times I_{22}\times I_{12}.
\end{align*}
Also let
\[M_n = c+J_n\subseteq L_n,\]
where $c=(n+2,\ldots,n+2,0)\in \Z^n$ and $J_n$ is defined as in \eqref{eq:jn}. Let $L_{n,0}=\partial L_n$, and for $k=1,\ldots,n+2$ let $L_{n,k}=L_{n,k-1}^\circ$. For each $k=0,\ldots,n+2$, let $B_{n,k}=\partial L_{n,k}$. For example, the graphs $B_{2,k}$ are outlined in Figure \ref{fig:c2}, and $M_2$ is the large outlined region in Figure \ref{fig:wrapping}. Note that the $B_{n,k}$ partition $L_n$ in a particular way; there is a ``central'' grid graph $B_{n,n+2}$ induced by all vertices of $M_n$ with $n$th coordinate $n+2$, and $B_{n,k}$ is ``wrapped around'' $B_{n,k+1}$ for $k=0,\ldots,n+1$.

Suppose that $x\in B_{n,k}$. Let $f_n:(L_n,\partial L_n)\to (G_n,a_{0,\ldots,0})$ be given by
\[f_n(x)=\gamma_n(0,x_2^*,x_3^*,\ldots,x_{n-1}^*,k),\]
where for $1<i<n$ we set
\[x_i^* = \begin{cases} x_i-(n+2)&\mbox{if } n+2\leq x_i\leq n+i+4\\ 0&\mbox{if }x_i\leq n+2\\ i+2 &\mbox{if } x_i\geq n+i+4.\end{cases}\]
For example, $f_2$ is illustrated in Figure \ref{fig:c2}.

We claim that $f_n$ is a graph map. Let $x,y\in L_n$ with $x\simeq y$. We must show that $f_n(x)\simeq f_n(x+y)$. If $x=y$, then this is immediate. Suppose that $x\neq y$ with $x\in B_{n,k}$. The vertices $x$ and $y$ differ in exactly one coordinate; call it $j$. We must have either $y\in B_{n,k}$ or $y\in B_{n,k\pm 1}$. If $j\in \{1,n\}$ and $y\in B_{n,k}$, then $f_n(x)=f_n(y)$ by definition. If $j\in \{1,n\}$ and $y\in B_{n,k\pm 1}$, then
\begin{equation*}
f_n(x) = \gamma_n(0,x_2^*,\ldots,x_{n-1}^*,k) \simeq \gamma_n(0,x_2^*,\ldots,x_{n-1}^*,k\pm 1) = f_n(y).
\label{eq:vw}
\end{equation*}
If $1<j<n$ and $y\in B_{n,k}$, then $|x_j^*-y_j^*|\leq 1$, so
\begin{equation}f_n(x)=\gamma_n(0,x_2^*,\ldots,x_{n-1}^*,k)\simeq \gamma_n(0,x_2^*,\ldots,x_{j-1}^*,y_j^*,x_{j+1}^*,\ldots,x_{n-1}^*,k)=f_n(y).
\label{eq:xtoy}
\end{equation}
If $1<j<n$ and $y\in B_{n,k\pm 1}$, then either $x_j,y_j\leq n+2$ or $x_j,y_j\geq n+j+4$, so $x_j^*=y_j^*$ and \eqref{eq:xtoy} holds again. Thus in any case we have $f_n(x)\simeq f_n(y)$, so $f_n$ is a graph map.

We claim that the following defines a function $g_n:(L_n,\partial L_n)\to (G_n,a_{0,\ldots,0})$:
\[g_n(x) = \begin{cases} \gamma_n(x - c)&\mbox{if }x\in M_n\\ f_n(x)&\mbox{if }x\in L_n\setminus M_n^\circ.\end{cases}\]
For example, $g_2$ is illustrated in Figure \ref{fig:wrapping} with $M_2$ boxed in black and $B_{2,k}\setminus M_2$ boxed in orange, green, blue and red for $k=0,\ldots,3$, respectively. To prove that $g_n$ is well defined in general, we must show that
\begin{equation}
\gamma_n(x-c)=f_n(x)
\label{eq:gammaf}
\end{equation}
for all $x\in \partial M_n$. It will follow that $g_n$ is a graph map, since $\gamma$ and $f_n$ are graph maps.

We have $x\in M_n$ if and only if the following hold:
\begin{enumerate}[(i)]
\item $n+2\leq x_1\leq n+7$
\item $n+2\leq x_i\leq n+i+4$ for $i=2,\ldots, n-1$
\item $0\leq x_n\leq n+2$.
\end{enumerate}
Thus $x\in \partial M_n$ if and only if (i)--(iii) hold with at least one of these inequalities being an equality. Note that the graphs $B_{n,k}\cap \partial M_n$ partition $\partial M_n$. Suppose that $x\in B_{n,k}\cap \partial M_n$. We have $x_n=k$ and $x_i^*=x_i-(n+2)$ for $1<i<n$, so
\[\gamma_n(x-c)=\gamma_n(x_1-(n+2),x_2^*,\ldots,x_{n-1}^*,k).\]
If $x_1 = n+2$, then
\[\gamma_n(x-c) = \gamma_n(0,x_2^*,\ldots,x_{n-1}^*,k)=f_n(x)\]
by definition. If $x_1=n+7$, then
\[\gamma_n(x-c) = \gamma_n(5,x_2^*,\ldots,x_{n-1}^*,k)=\gamma_n(0,x_2^*,\ldots,x_{n-1}^*,k)=f_n(x),\]
where we have used \eqref{eq:gflat05}. If $x_i = n+2$ for some $1<i<n$, then $x_i^*=0$, so
\begin{align*}
\gamma_n(x-c) &= \gamma_n(x_1-(n+2),x_2^*,\ldots,x_{i-1}^*,0,x_{i+1}^*,\ldots,x_{n-1}^*,k)\\
&=\gamma_n(0,\ldots,0,x_{i+1}^*,\ldots,x_{n-1}^*,k)\\
&=\gamma_n(0,x_2^*,\ldots,x_{i-1}^*,0,x_{i+1}^*,\ldots,x_{n-1}^*,x_n)\\
&=\gamma_n(0,x_2^*,\ldots,x_{n-1}^*,k)\\
&=f_n(x).
\end{align*}
where we have used \eqref{eq:gflat}. The argument is similar for $x_i=n+i+4$, using \eqref{eq:gflat2} instead. If $x_n=0$ (resp., $x_n=n+2$), then \eqref{eq:gflat} (resp., \eqref{eq:gflat2}) again implies \eqref{eq:gammaf}. Therefore \eqref{eq:gammaf} holds for all $x\in \partial M_n$, proving the claim that $g_n$ is well defined. It follows that $g_n$ is a graph map.

We now prove \eqref{eq:psigf}. Let
\[\Omega_n=\{x\in \Z^n : x+Q_n \subseteq M_n\}.\]
Since $g_n(x)=f_n(x)$ for all $x\in L_n\setminus M_n^\circ$, we have
\[\phi([g_n] - [f_n]) = \sum_{x\in \Omega_n} g_n^x - \sum_{x\in \Omega_n} f_n^x.\]
The first sum on the right hand side is easily seen to be $\phi(\gamma_n)$. We claim that every term of the second sum is a degenerate $n$-cube of $f_n$. If $x\in \Omega_n$, then $n+2\leq x_i<n+i+4$ for all $i<n$, so for any $q\in Q_n$ we have
\begin{align*}
f_n^x(0,q_2,\ldots,q_n) &= \gamma_n(0,x_2+q_2-(n+2),\ldots,x_{n-1}+q_{n-1}-(n+2),x_n+q_n)\\
&=f_n^x(1,q_2,\ldots,q_n).
\end{align*}
Hence $D_1^-f_n^x = D_1^+f_n^x$, proving the claim. It follows that
\[\phi([g_n]-[f_n]) = \phi(\gamma_n),\]
proving \eqref{eq:psigf}.
\end{proof}

\begin{remark}
Close inspection reveals that the image of $f_n$ is the set of all vertices of $G_n$ of the form $a_{i_1,\ldots,i_{n-1}}$, where we have added subscripts to the vertex $a\in G_1$. This set induces a subgraph of $G_n$ isomorphic to the graph $U_n$, where $U_1$ consists of a single vertex and $U_{n+1}=S_{n+3}U_n$ for all $n\geq 1$. Proposition \ref{prop:suscon} implies that $U_n$ is contractible for all $n$. Hence $[f_n]=0$ in $A_n(G_n)$, so in fact $\psi([g_n])=\overline{\phi(\gamma_n)}$.
\end{remark}

\begin{proof}[{Proof of Theorem \ref{thm:hursurj}}]
Let $G_1$ be any graph containing no 3- or 4-cycles. If $G_1$ contains no cycles, then $G_n$ is contractible for all $n$ by Example \ref{eg:acyclic} and Proposition \ref{prop:suscon}, so $\hc_n(G_n)$ is trivial for all $n\geq 1$ by \cite[Lemma 4.2]{barcelo2019-1}, and the theorem is immediate.

Suppose that $G_1$ contains a cycle. If $k<n$, then $\hc_k(G_n)=0$ by Proposition \ref{prop:homgen}(i), and the theorem is immediate. We prove the case $k=n$. Let $\{Z_{1,1},\ldots,Z_{1,\ell}\}$ be the set of cycles of $G_1$, considered as subgraphs of $G_1$. For $i=1,\ldots,\ell$ and $n\geq 1$, let $Z_{n+1,i} = S_{n+3}Z_{n,i}$, considered as a subgraph of $G_{n+1}$. Lemma \ref{lem:hursurj1} can be easily generalized to obtain grid graphs $J_{n,i}$ and graph maps $\gamma_{n,i}:J_{n,i}\to Z_{n,i}$ satisfying the following for all $i$ and $n\geq 1$:
\begin{enumerate}[(i)]
\item $\hc_n(Z_{n,i})=\langle \overline{\phi(\gamma_{n,i})}\rangle$
\item $\Delta(\overline{\phi(\gamma_{n+1,i})})=(-1)^n \overline{\phi(\gamma_{n,i})}$, where $\Delta:\hc_{n+1}(G_{n+1})\to \hc_n(G_n)$ is the isomorphism from Proposition \ref{prop:homgen}(ii).
\end{enumerate}
Let
\[\Upsilon_n=\{\overline{\phi(\gamma_{n,i})}:i=1,\ldots,\ell\}.\]
Example \ref{eg:homsus} and item (i) together say that $\hc_1(G_1)$ is generated by $\Upsilon_1$. Item (ii) then implies that $\hc_n(G_n)$ is generated by $\Upsilon_n$ for all $n\geq 1$. Lemma \ref{lem:hursurj1} can be easily generalized to show that the image of $\psi:A_n(G_n)\to \hc_n(G_n)$ contains $\Upsilon_n$. Hence $\psi$ is surjective.
\end{proof}
\section{Final remarks}
\label{sec:openq}

Our work leaves open several important questions. Let $G_1$ be a cycle graph of length 5, and define the graphs $G_n$ as in \eqref{eq:gndef}. Our most immediate goal is to show that the map $\psi:A_n(G_n)\to \hc_n(G_n)$ is an isomorphism for all $n$. Theorem \ref{thm:hursurj} brings us halfway there; we leave the remaining half as a conjecture.

\begin{conj}
The map $\psi:A_n(G_n)\to \hc_n(G_n)$ is injective for all $n$.
\label{conj:inj}
\end{conj}

The main analogy guiding our intuition is to think of the graphs $G_n$ as playing the role of the $n$-sphere $S^n$ in classical topology. An obvious comparison to make is that $S^{n+1}\approx S S^n$, where $S$ is the usual suspension functor, since this mirrors the construction of $G_n$ via the discrete suspension functor $S_t$. Our analogy is strengthened by the fact that
\[\hc_i(G_n)\cong H_i(S^n)\]
for all $i\leq n$. An important property of $S^n$, however, is $(n-1)$-connectedness. This is what allows one to apply the Hurewicz theorem and conclude that $\pi_n(S^n)\cong H_n(S^n)$. We suspect that $G_n$ is $(n-1)$-connected in a discrete sense.

\begin{conj}
If $i<n$, then $A_i(G_n)$ is trivial.
\end{conj}

The usual way to prove that $S^n$ is $(n-1)$-connected is to invoke the cellular approximation theorem. A discrete analog of cellular approximation was proposed in \cite{babson2006}. An alternative proof of $(n-1)$-connectedness, such as the one in \cite[Theorem 6.4.4]{tomdieck2008}, uses the \emph{homotopy excision theorem} of Blakers and Massey.

\begin{thm}[Blakers-Massey {\cite{blakers1952}}]
Suppose that a topological space $X$ is the union of open subspaces $A$ and $B$ with nonempty intersection $C=A\cap B$. If $(A,C)$ is $m$-connected and $(B,C)$ is $n$-connected, then the map
\[\pi_k(B,C)\to \pi_k(X,A),\]
induced by the inclusion $(B,C)\to (X,A)$, is surjective if $k\leq m+n$ and bijective if $k<m+n$.
\label{thm:blakers}
\end{thm}

This result is a powerful means of computing higher classical homotopy groups. In particular, it is a key ingredient in the proof of the Hurewicz theorem. We therefore expect some discrete version of Theorem \ref{thm:blakers} to appear in a proof of Conjecture \ref{conj:hur}, should it hold. At the very least, a discrete homotopy excision theorem would enable us to perform more elegant computations than are currently possible.

\begin{q}
Is there a discrete homotopy excision theorem?
\end{q}

A famous corollary of homotopy excision is the \emph{Freudenthal suspension theorem}, which states that if $X$ is $n$-connected, then the map
\[\pi_k(X)\to \pi_{k+1}(\Sigma X),\]
induced by suspension, is an isomorphism for $k<2n+1$ and surjective if $k=2n+1$. This gives, for example, a sequence of maps
\[\pi_1(S^1)\to \pi_2(S^2)\to \pi_3(S^3)\to \cdots\]
in which the first map is surjective and all subsequent maps are isomorphisms. Ideally, a discrete homotopy excision theorem would give such a sequence for the graphs $G_n$, i.e. a sequence of maps
\[A_1(G_1)\to A_2(G_2)\to A_3(G_3)\to \cdots\]
in which the first map is surjective and the rest are isomorphisms. Since $A_1(G_1)\cong \Z$ and $A_2(G_2)$ is infinite by Corollary \ref{cor:hursurj}, the first map must be an isomorphism as well. Thus it would follow that $A_n(G_n)\cong \Z$ for all $n$, confirming Conjecture \ref{conj:inj} and further reinforcing the analogy between $G_n$ and $S^n$.

\section*{Acknowledgments}

The author thanks H\'{e}l\`{e}ne Barcelo and Curtis Greene for extensive discussions and suggestions. The author also thanks Trevor Hyde for helpful comments on a draft of the paper.

\bibliographystyle{abbrv}
\bibliography{lutz-hdhog}
\end{document}